\documentclass[12pt,a4paper,twoside,reqno]{amsart}

\usepackage[all,cmtip]{xy}
\usepackage{tikz}
\usepackage{amsmath,amscd}
\usepackage{amssymb,amsfonts,mathrsfs}
\usepackage[driver=pdftex,margin=3cm,heightrounded=true,centering]{geometry}
\usepackage{JK}
\usepackage[colorlinks=true,linkcolor=blue,citecolor=blue]{hyperref}

\usepackage[all]{xy}
\newtheorem{thm}{Theorem}[section]

\usepackage{graphicx}

\begin{document}
\title[]{A Serre-Swan theorem for bundles of bounded geometry}

\author{Jens Kaad}
\address{Institut de Math\'ematiques de Jussieu,
Universit\'e de Paris VII,
175 rue du Chevaleret,
75013 Paris,
France}
\email{jenskaad@hotmail.com}

%
%
%
\thanks{The author is supported by the Fondation Sciences Math\'ematiques de Paris.}
\subjclass[2010]{47L25, 53C20; 47L30, 53B20}
\keywords{Bounded geometry, Hilbert bundles, Operator $*$-algebras, Operator $*$-modules, Serre-Swan Theorem.}

\begin{abstract}
The Serre-Swan theorem in differential geometry establishes an equivalence between the category of smooth vector bundles over a smooth compact manifold and the category of finitely generated projective modules over the unital ring of smooth functions.

This theorem is here generalized to manifolds of bounded geometry. In this context it states that the category of Hilbert bundles of bounded geometry is equivalent to the category of operator $*$-modules over the operator $*$-algebra of continuously differentiable functions which vanish at infinity.

Operator $*$-modules are generalizations of Hilbert $C^*$-modules where $C^*$-algebras have been replaced by a more flexible class of involutive algebras of bounded operators: Operator $*$-algebras. They play an important role in the study of the unbounded Kasparov product.
\end{abstract}

\maketitle
\tableofcontents
\section{Introduction}
In differential geometry the Serre-Swan theorem for a smooth compact manifold $\C M$ states that the category of smooth vector bundles over $\C M$ is equivalent to the category of finitely generated projective modules over the unital commutative ring of smooth functions $C^\infty(\C M)$. The desirable functor sends a smooth vector bundle $E \to \C M$ to the $C^\infty(\C M)$-module of smooth sections $\Ga^\infty(E)$. See \cite[Section 3.2]{VaGr:CDS}, \cite[Chapter I, Section 6.18]{Kar:KT}, \cite[Theorem 2]{Swa:VBP}. 
%

The present paper investigates the two basic questions:
\vspace{5pt}

\emph{What happens when the compactness condition on the manifold $\C M$ is relaxed? And what if vector bundles are replaced by Hilbert bundles?}
\vspace{5pt}

From one point of view, the first question has already been answered. In a short note, A. Morye establishes general conditions on a locally ringed space $(X,\C O_X)$ which imply that the global section functor is an equivalence between the category of locally free modules of bounded rank over the sheaf $\C O_X$ and the finitely generated projective modules over the ring $\C O_X(X)$. See \cite[Theorem 2.1]{Mor:NST}. The conditions do in particular apply to any smooth manifold $\C M$ equipped with the sheaf of smooth functions. This means that the category of smooth vector bundles over $\C M$ is equivalent to the category of finitely generated projective modules over the smooth functions $C^\infty(\C M)$. The main properties needed in this example are, that $\C M$ has finite covering dimension, and that any open cover $\{U_i\}_{i \in I}$ admits a smooth partition of unity $\{\chi_i\}_{i \in I}$ with $\T{supp}(\chi_i) \su U_i$ for all $i \in I$ (thus $\{\chi_i\}$ is subordinate to $\{U_i\}$). See also the short note of G. Sardanashvily, \cite{Sar:SSN}, for a more elementary approach.
%

Suppose now that the manifold $\C M$ is Riemannian. An interesting ring to work with then consists of the continuously differentiable functions which \emph{vanish at infinity} (the de Rham differential is also assumed to vanish at infinity). A Serre-Swan theorem in this context could therefore aim for a characterization of the finitely generated projective modules over $C^1_0(\C M)$ in terms of differentiable vector bundles over $\C M$. It does however soon become apparent that the non-compactness of $\C M$ requires us to pass from finitely generated modules to a suitable class of countably generated modules. Indeed, contrary to the case of the ring $C^\infty(\C M)$ discussed above, the existence of a \emph{subordinate} partition of unity $\{\chi_i\}$ fails when the functions $\chi_i$ are required to vanish at infinity. The importance of the passage to the countably generated setup becomes even more apparent when the aim is to incorporate Hilbert bundles. These considerations raise the following question:
\vspace{5pt}

\emph{Which kind of countably generated modules over $C^1_0(\C M)$ is expected to appear in a Serre-Swan theorem? Thus, which kind of modules corresponds to differentiable Hilbert bundles on non-compact Riemannian manifolds?}
\vspace{5pt}

A source of inspiration for answering these questions is Kasparov's stabilization theorem. The content of this theorem is that any countably generated Hilbert $C^*$-module over any $C^*$-algebra $A$ is an orthogonal direct summand in a canonical "free" module over $A$. Notice that countably generated is meant in a topological sense, thus the requirement is that a dense submodule is algebraically countably generated. See \cite[Theorem 2]{Kas:HSV}, \cite[Theorem 1.4]{MiPh:ETH}, \cite[Theorem 13.6.2]{Bla:KOA}. When $A$ consists of the \emph{continuous} functions vanishing at infinity on a manifold $\C M$, the stabilization theorem implies the following: For any separable continuous field of Hilbert spaces $\sH$ on $\C M$ there exists a strongly continuous projection valued map $\sP : \C M \to \sL(H)$ and an isomorphism of modules $\Ga_0(\sH) \cong \sP C_0(\C M,H)$ over $C_0(\C M)$. Here $\sL(H)$ are the bounded operators on a separable Hilbert space $H$, and $\Ga_0(\sH)$ are the continuous sections of the field which vanish at infinity. The module $\sP C_0(\C M,H)$ consists of the continuous maps $f : \C M \to H$ which vanish at infinity with $f(x) \in \T{Im}(\sP(x))$ for all $x \in \C M$. See \cite[Chapter 2, Appendix A]{Con:NCG} and \cite{DiDo:CCH}.
%

In view of this result one could consider modules $X$ over $C^1_0(\C M)$ for which there exist a strongly differentiable projection valued map $\sP : \C M \to \sL(H)$ and an isomorphism $X \cong \sP C^1_0(\C M,H)$ of modules over $C^1_0(\C M)$. Here $C^1_0(\C M,H)$ are the $C^1$-functions which vanish at infinity and take values in a separable Hilbert space $H$. We will furthermore require an upper bound on the strong de Rham derivative of $\sP$.

The importance of this kind of modules is underpinned by the abstract framework of operator $*$-modules which was invented in \cite{KaLe:SFU} as an important tool for the construction of the unbounded version of the Kasparov product in $KK$-theory. In a geometric framework an operator $*$-module serves as a domain for a canonical Gra\ss mann connection. See also \cite{Mes:UCN}. An operator $*$-module can be thought of as an analogue of a Hilbert $C^*$-module, but where the $C^*$-algebra has been replaced by a more flexible involutive algebra of bounded operators called an operator $*$-algebra. 

A (concrete) operator $*$-algebra can be shortly defined as a closed subalgebra $A \su \sL(H)$ which comes equipped with a completely bounded involution $\da : A \to A$. The involution $\da$ is typically different from the adjoint operation. Since $A$ is an operator algebra it has an associated (column) standard module $A^\infty$ and the existence of the completely bounded involution implies that this standard module has a canonical completely bounded $A$-valued hermitian form. When $A$ is a $C^*$-algebra this construction recovers both the standard module $H_A$ and the usual $A$-valued hermitian form. See \cite{Ble:GHM}.

In view of Kasparov's stabilization theorem, it is now natural to take a careful look at the orthogonal direct summands in the standard module $A^\infty$ over an operator $*$-algebra $A$. These are by definition the operator $*$-modules over $A$. The morphisms are the completely bounded maps which have a completely bounded adjoint with respect to the canonical hermitian forms.

The concept of an operator $*$-module is thus strongly related to D. Blecher's notion of a (CCGP)-module over an operator algebra $A$ (and thus to his notion of a rigged module). See \cite[Definition 8.1]{Ble:GHM}. The incorporation of completely bounded involutions and the associated completely bounded hermitian forms seems however to be novel. Another difference is that our work is firmly rooted in the "completely bounded" setup and not the "completely contractive" setup of \cite{Ble:GHM}.

The present paper is concerned with the $*$-algebra of continuously differentiable functions on a Riemannian manifold which vanish at infinity. This $*$-algebra becomes a (concrete) operator $*$-algebra via the injective algebra homomorphism
\[
C^1_0(\C M) \to \sL\big(L^2(\La(T^*\C M) \op L^2(\La(T^*\C M)) \big) \q 
f \mapsto \ma{cc}{f & 0 \\ df & f},
\]
where $L^2(\La(T^*\C M))$ denotes the Hilbert space of square integrable forms and $df \in \Ga_0(T^*\C M)$ is the de Rham derivative. See \cite[Proposition 2.8]{KaLe:SFU}. A similar observation implies that any spectral triple has an associated canonical operator $*$-algebra, \cite[Proposition 2.6]{KaLe:SFU}, \cite[Section 3.1]{Mes:UCN}.
%
%

In the case of $C^1_0(\C M)$, the standard module $C^1_0(\C M)^\infty$ is isomorphic to the Hilbert space valued continuously differentiable functions which vanish at infinity, $C^1_0(\C M,H)$. The canonical completely bounded hermitian form is simply given by $\inn{s,t} : x \mapsto \inn{s(x),t(x)}_H$. See \cite[Proposition 3.6]{KaLe:SFU}.

The first important result of this paper describes the endomorphism ring of the standard module $C^1_0(\C M)^\infty$ when $\C M$ is complete.

\begin{prop}
Suppose that $\C M$ is a complete Riemannian manifold. Then there is a canonical bijective correspondence between the endomorphism ring $\T{End}(C^1_0(\C M)^\infty)$ and the $*$-strongly differentiable maps $\al : \C M \to \sL(H)$ with 
\[
\sup_{x \in \C M}\big( \|(d\al)(x)\|_\infty + \|\al(x)\|_\infty \big) < \infty,
\]
where $(d\al)(x) : H \to H \ot T_x^*(\C M)$ denotes the strong de Rham derivative at a point $x \in \C M$.
\end{prop}

A consequence of this proposition is that operator $*$-modules over $C^1_0(\C M)$ correspond precisely to strongly differentiable projection valued maps $P : \C M \to \sL(H)$ with $\sup_{x \in \C M}\|(dP)(x)\|_\infty < \infty$. The "natural" class of modules alluded to above, after the discussion of continuous fields of Hilbert spaces, thus appears in a \emph{canonical way} when the operator $*$-algebra structure on $C^1_0(\C M)$ is fixed.

Our initial questions on the Serre-Swan theorem can now be sensibly rephrased as follows:
\vspace{5pt}

\emph{Can we characterize the operator $*$-modules over $C^1_0(\C M)$ in terms of differentiable Hilbert bundles over the Riemannian manifold $\C M$?}
\vspace{5pt}

See also \cite[Remark 3.7]{KaLe:SFU}.

When dealing with this question on the Serre-Swan theorem it soon becomes apparent that some further restrictions on the Riemannian manifold $\C M$ are needed. It turns out that a suitable condition is a weak form of bounded geometry referred to by J. Cheeger, M. Gromov, and M. Taylor as $C^0$-bounded geometry. This condition means that the injectivity radius $r_{\T{inj}}$ of the Riemannian manifold is strictly positive and that there exists a constant $r \in (0,r_{\T{inj}})$ such that the derivatives of the exponential maps $\exp_x : B_r(0) \to U_x \su \C M$ and their inverses are uniformly bounded. Here $B_r(0) \su (T_x \C M)_{\rr}$ denotes the open ball of radius $r$ in the tangent space over a point $x \in \C M$. See \cite[Section 3]{CGT:FKL}.

Similarly, the Hilbert bundles on the geometric side of the Serre-Swan theorem are required to be of $C^0$-bounded geometry. This means that the Hilbert bundle $\sH \to \C M$ can be trivialized over each normal coordinate neighborhood $\exp_x(B_r(0)) = U_x \su \C M$ such that the transition maps $\tau_{x,y} : U_x \cap U_y \to \sL(H)$ are strongly differentiable and take values in the group of unitaries. Furthermore, the strong de Rham derivatives $(d\tau_{x,y})(z) : H \to H \ot T_z(\C M)$ are required to be uniformly bounded in all parameters. A morphism of Hilbert bundles of $C^0$-bounded geometry is a morphism of the underlying Hilbert bundles such that the transition maps $\al_{x,y} : U_x\cap U_y \to \sL(H,G)$ are $*$-strongly differentiable with uniformly bounded strong derivatives.

When $\sH \to \C M$ is a vector bundle this notion of bounded geometry is a $C^0$-version of the one appearing in the work of M. Shubin for example, see \cite[Section A1.1]{Shu:SEN}.
%

The main result of the present text can now be formulated precisely:

\begin{thm}
Let $\C M$ be a manifold of bounded geometry. The category of Hilbert bundles of bounded geometry over $\C M$ is equivalent to the category of operator $*$-modules over $C^1_0(\C M)$. 
\end{thm}

The equivalence is given by the functor $\Ga_0^1$ which sends a Hilbert bundle of bounded geometry to the module of continuously differentiable sections which vanish at infinity and a morphism $\al : \sH \to \sG$ to the homomorphism $\Ga_0^1(\al) : s \mapsto \al \ci s$. 
\vspace{1pt}

The current text is organized as follows:

In Section \ref{s:manfol}, we recall the notion of $C^0$-bounded geometry and introduce the category of Hilbert bundles of bounded geometry. Some references for this section are \cite{Eic:GOM}, \cite{CGT:FKL}, \cite{Shu:SEN}, \cite{Roe:ITO}.

In Section \ref{s:opemodI}, we recall the definition of an operator $*$-module over an operator $*$-algebra. It is in this respect necessary to review the basic definitions from the theory of operator spaces, operator algebras, and operator modules. The section also contains a study of the operator $*$-algebra of $C^1$-functions on a Riemannian manifold which vanish at infinity. This involves the computation of the standard module and the associated endomorphism ring. Some references for this section are \cite{Pis:IOT}, \cite{BlMe:OMO}, \cite{Pau:CBO}, \cite{BlMuPa:CMP}, \cite{Ble:GHM}, \cite{Ble:AHM}.

The last three sections contain the main contribution of the present paper. 

In Section \ref{s:difstab} it is established that the differentiable sections which vanish at infinity of a Hilbert bundle of bounded geometry is an operator $*$-module over $C^1_0(\C M)$.

In Section \ref{s:imabun} it is proved that an operator $*$-module over $C^1_0(\C M)$ has an associated image Hilbert bundle of bounded geometry. 
%

In the final Section \ref{s:serswa} the above results are combined to a proof of the Serre-Swan theorem. In particular it is established that any operator $*$-module is isomorphic to the $C^1$-sections which vanish at infinity of a Hilbert bundle of bounded geometry.

\section*{Acknowledgements} I would like to thank Matthias Lesch for the good discussions we had on the subject of the present paper at many occasions.


%
%

\section{Bounded geometry}\label{s:bougeo}

\subsection{Manifolds}\label{s:manfol}
Throughout this section $\C M$ will denote a smooth manifold of dimension $N \in \nn$. Thus, $\C M$ is a connected, second countable, Hausdorff topological space with a maximal atlas of smooth charts $\sF$.

Let $T\C M \to \C M$ and $T^*\C M \to \C M$ denote the complexified tangent bundle and the complexified cotangent bundle over the smooth manifold $\C M$. The smooth sections $\Ga^\infty(T\C M)$ and $\Ga^\infty(T^*\C M)$ are referred to as smooth vector fields and smooth $1$-forms. They are modules over the unital ring $C^\infty(\C M)$ of smooth complex valued functions on $\C M$.

For any smooth map $f : \C M \to \C N$, let $df(x) : T_x\C M \to T_{f(x)} \C N$ denote the derivative at a point $x \in \C M$.
%

The notation $d : C^\infty(\C M) \to \Ga^\infty(T^*\C M)$ refers to the de Rham differential. Thus explicitly in a smooth chart $(\phi,U) \in \sF$ we have that
\[
(df)(x) = \sum_{i=1}^N \frac{\pa f}{\pa \phi_i}\big|_x \cd (d\phi_i)(x) \in T_x^*\C M
\]
for all $x \in U$ and all smooth functions $f : \C M \to \cc$, where $(d\phi_i)(x) \in T_x^*\C M$ are the dual basis vectors to $\frac{\pa}{\pa \phi_i}\big|_x \in T_x \C M$, $i \in \{1,\ldots,N\}$.
%


\begin{dfn}
A smooth manifold $\C M$ is \emph{Riemannian} when it is equipped with an inner product $\inn{\cd,\cd}_x : T_x\C M \ti T_x\C M \to \cc$ for each $x \in \C M$. Furthermore, the map $\binn{\frac{\pa}{\pa \phi_i},\frac{\pa}{\pa \phi_j}} : U \to \cc$ is required to be real valued and smooth whenever $\phi_i,\phi_j : U \to \rr$ are coordinate functions of a smooth chart $(\phi,U) \in \sF$.
\end{dfn}


Let $\C M$ be a smooth Riemannian manifold. The hermitian form $\inn{\cd,\cd} : \Ga^\infty(T\C M) \ti \Ga^\infty(T\C M) \to C^\infty(\C M)$ establishes an isomorphism $\flat : \Ga^\infty(T\C M) \to \Ga^\infty(T^*\C M)$ given by $X^\flat : Y \mapsto \inn{X,Y}$. In particular, there is an associated hermitian form $\inn{\cd,\cd} : \Ga^\infty(T^*\C M) \ti \Ga^\infty(T^*\C M) \to C^\infty(\C M)$, $\inn{X^\flat,Y^\flat} := \inn{Y,X}$. 
%
%

Let $(T\C M)_\rr \to \C M$ denote the tangent bundle over $\C M$. Let $x \in \C M$. Recall then that there exists an $r > 0$ such that the exponential map $\exp_x : B_r(0) \to U_{x,r}$ is a diffeomorphism, where $B_r(0) \su (T_x \C M)_\rr$ is the open ball of radius $r >0$ and center $0$ and $U_{x,r} \su \C M$ is an open subset. See \cite[Chapter III, Proposition 8.1]{KoNo:FDGI}.

Each choice of orthonormal basis for $(T_x \C M)_\rr$ therefore gives rise to a smooth chart $\phi_{x,r} : U_{x,r} \to \rr^N$. Such a smooth chart will be called \emph{normal}. Notice that the image $\phi_{x,r}(U_{x,r})$ is the open ball $B_r(0)$ in $\rr^N$ with radius $r$ and center $0$.

\begin{dfn}\label{d:bougeo}
A smooth Riemannian manifold is of \emph{bounded geometry} when
\begin{enumerate}
\item There exists an $r > 0$ such that $\exp_x : B_r(0) \to U_{x,r}$ is a diffeomorphism for all $x \in \C M$.
\item There exists an $r > 0$ satisfying the above condition and a constant $C > 0$ such that
\[
\begin{split}
\T{sup}_{y \in U_{x,r}}\|d\phi_{x,r}(y)\|_\infty \leq C \q \T{and} \q
\T{sup}_{v \in B_r(0)}\|d\phi^{-1}_{x,r}(v)\|_\infty \leq C
\end{split}
\]
for all $x \in \C M$. Here $(d\phi_{x,r})(y) : T_y\C M \to T_{\phi_{x,r}(y)}\rr^N$ and its inverse are perceived as bounded operators. The norms appearing are thus operator norms.
%
\end{enumerate}
%
\end{dfn}

\begin{remark}\label{r:bougeoI}
The first condition in the above definition means precisely that the \emph{injectivity radius} $r_{\T{inj}}$ of $\C M$ is strictly positive. 

%

For each smooth chart $\phi : U \to \rr^N$, let $g_{\phi} : U \to \T{GL}_N(\rr)$ be defined by $(g_{\phi})_{ij} := \binn{\frac{\pa}{\pa \phi_i}, 
\frac{\pa}{\pa \phi_j}}$. The second condition is then equivalent to the existence of constants $C > 0$ and $r \in (0,r_{\T{inf}})$ such that
\[
\T{sup}_{y \in U_{x,r}}\|g_{\phi_{x,r}}(y)\|_\infty \leq C \q \T{and} \q
\T{sup}_{y \in U_{x,r}}\|g_{\phi_{x,r}}^{-1}(y)\|_\infty \leq C
\]
for all $x \in \C M$. The norms appearing are again operator norms since each $g_{\phi_{x,r}}(y)$ can be perceived as a bounded operator on $\cc^N$.

The notion of bounded geometry used here is thus equivalent to $C^0$-bounded geometry as defined in \cite[Section 3]{CGT:FKL}.

%
\end{remark}


The next two general lemmas on manifolds of bounded geometry will suffice for the purposes of this paper. They are variations of results appearing in \cite[Appendix A1.1]{Shu:SEN}.

The open balls are with respect to the metric $d_{\C M} : \C M \ti \C M \to [0,\infty)$ associated with the Riemannian structure. Note also that a countable open cover $\{U_i\}_{i \in \nn}$ of $\C M$ is of \emph{finite multiplicity} when there exists a $K \in \nn$ such that $\bigcap_{j \in J} U_j \neq \emptyset \Rightarrow |J| \leq K$. Here $|J|$ is the number of elements in the subset $J \su I$.

\begin{lemma}\label{l:finmul}
Let $\C M$ be a manifold of bounded geometry. There exists an $\ep_0 > 0$ such that for each $\ep \in (0,\ep_0)$ there exists a countable cover $\{B_{\ep}(y_i)\}_{i \in \nn}$ of $\C M$ such that the countable cover $\{B_{2\ep}(y_i)\}_{i \in \nn}$ has finite multiplicity. Here $y_i \in \C M$ is the center and $\ep > 0$ is the radius of the open ball $B_\ep(y_i) \su \C M$.
\end{lemma}

\begin{lemma}\label{l:paruni}
Let $\{B_{\ep}(y_i)\}_{i \in \nn}$ be an open cover as in the above lemma. Then there exists a partition of unity $\{\chi_i\}$ such that each $\chi_i : \C M \to [0,1]$ has a smooth square root with $\T{supp}(\sqrt{\chi_i}) \su B_{2\ep}(y_i)$. Furthermore, $\T{sup}_{i \in \nn}\|d\sqrt{\chi_i}\|_\infty < \infty$.
\end{lemma}

Notice that the norm $\|d\sqrt{\chi_i}\|_\infty$ in the above lemma is the supremum norm, $\|d\sqrt{\chi_i}\|_\infty := \sup_{x \in \C M}\|(d\sqrt{\chi_i})(x)\|_\infty$.

\begin{remark}\label{r:comman}
Each of the functions $\chi_i : \C M \to [0,1]$ from the above lemma has compact support. Indeed, the bounded geometry condition on $\C M$ entails that $\C M$ is a complete manifold, see \cite[Proposition 1.2a]{Eic:GOM}. But this implies that the closed ball $\T{cl}\big( B_{\ep}(y_i) \big) \su \C M$ is compact. See \cite[Chapter IV, Theorem 4.1]{KoNo:FDGI}
\end{remark}

%

The $*$-algebra of continuously differentiable function which vanish at infinity on a smooth Riemannian manifold will play an important role in this paper. To avoid any confusion we give a precise definition.

\begin{dfn}
Let $\C M$ be a smooth Riemannian manifold. A function $f : \C M \to \cc$ is \emph{continuously differentiable} (or $C^1$) when the derivatives $\frac{\pa f}{\pa \phi_i} : U \to \cc$, $i \in \{1,\ldots,n\}$, exist and are continuous for any smooth chart $(\phi,U) \in \sF$. A $C^1$-function \emph{vanishes at infinity} when there for each $\ep > 0$ exists a compact set $\C K \su \C M$ such that
\[
\T{sup}_{x \in \C M\setminus \C K} \big( |f(x)| + \inn{df,df}^{1/2}(x) \big) < \ep.
\]

The $C^1$-functions which vanish at infinity form a $*$-algebra which is denoted by $C^1_0(\C M)$. The involution is given by complex conjugation.

A $C^1$-function which vanish at infinity will often be referred to as a \emph{$C^1_0$-function}.
\end{dfn}


\subsection{Hilbert bundles}\label{s:HilBun}
Throughout this section $H$ and $G$ will be \emph{separable} Hilbert spaces and $\C M$ will be a smooth Riemannian manifold of dimension $N \in \nn$. The notation $\sL(H,G)$ refers to the bounded linear operators from $H$ to $G$. The operator norm will be denoted by $\|\cd\|_\infty : \sL(H,G) \to [0,\infty)$ and the adjoint operation by $* : \sL(H,G) \to \sL(G,H)$. The norms on the Hilbert spaces are denoted by $\|\cd\|_H : H \to [0,\infty)$, etc.

\begin{dfn}\label{d:HilBunI}
A map $f : \C M \to H$ is \emph{continuously differentiable} (or $C^1$) when the derivatives $\frac{\pa f}{\pa \phi_i} : U \to H$, $i \in \{1,\ldots,n\}$, exist and are continuous for any smooth chart $(\phi,U) \in \sF$. A $C^1$-map \emph{vanishes at infinity} when there for each $\ep > 0$ exists a compact set $\C K \su \C M$ such that
\[
\T{sup}_{x \in \C M\setminus \C K} \big( \|f(x)\|_H + \|(df)(x)\|_{H \ot T_x^*\C M} \big) < \ep.
\]
Here $(df)(x) : H \to H \ot T_x^*\C M$ is defined in any smooth chart $(\phi,U)$ near $x$ by $(df)(x) := \sum_{i=1}^N \frac{\pa f}{\pa \phi_i}(x) \ot (d\phi_i)(x)$.

The $C^1$-maps which vanish at infinity form a right module over $C^1_0(\C M)$, where the action is given by pointwise scalar multiplication.

A $C^1$-map which vanish at infinity will often be referred to as a \emph{$C^1_0$-map}.
\end{dfn}


\begin{dfn}
A map $f : \C M \to \sL(H,G)$ is \emph{strongly differentiable} when $f(\xi) : \C M \to G$, $f(\xi)(x) := f(x)(\xi)$ is continuously differentiable for each $\xi \in H$.

The map $f : \C M \to \sL(H,G)$ is \emph{$*$-strongly differentiable} when it is strongly differentiable and when the \emph{adjoint} $f^* : x \mapsto f(x)^*$ is strongly differentiable as well.
\end{dfn}

\begin{remark}
The strong derivative of a strongly differentiable map, $df(x) : H \to G \ot T_x^*\C M$ defined by $df(x)(\xi) := d(f(\xi))(x)$, is a bounded operator for each $x \in \C M$. This is a consequence of the Banach-Steinhaus theorem since each $\frac{\pa f}{\pa \phi_i}(x) : H \to G$ is the strong limit of a sequence of bounded operators, where $(\phi,U)$ is a smooth chart near $x \in \C M$. See \cite[Theorem 2.8]{Rud:FA}.
\end{remark}

Before giving the main definition of this section, we present some preparatory lemmas on strongly differentiable maps.

Let $\phi : U \to \rr^N$ be a smooth chart and recall that $g_\phi : U \to GL_N(\rr) \su M_N(\cc)$ is the matrix valued map given by $(g_\phi)_{ij} = \inn{\frac{\pa}{\pa \phi_i},\frac{\pa}{\pa \phi_j}}$. In the following lemma, $g_\phi : U \to M_N(\cc)$ is perceived as a map with values in the $C^*$-algebra of linear operators $\sL(\cc^N,\cc^N)$.

\begin{lemma}\label{l:opecon}
Let $\phi : U \to \rr^N$ be a smooth chart with $\phi(U)$ convex and with $\sup_{z \in U}\|g_\phi(z)\|_\infty < \infty$. Suppose that $f : U \to \sL(H,G)$ is strongly differentiable with $\sup_{z \in U}\|(df)(z)\|_\infty < \infty$. Then
\[
\|f(x) - f(y)\|_\infty \leq \|\phi(x) - \phi(y)\|_{\rr^N} \cd \T{sup}_{z \in U}\|(df)(z)\|_{\infty} \cd \sup_{z \in U} \|g_\phi(z)\|_\infty^{1/2}
\]
for all $x,y \in U$, where $\|\cd\|_{\rr^N} : \rr^N \to [0,\infty)$ is the Euclidian norm. In particular $f : U \to \sL(H,G)$ is operator norm continuous.
\end{lemma}
\begin{proof}
To ease the notation, let $\ga_\phi := \sup_{z \in U}\|g_\phi(z)\|_\infty^{1/2}$.
%

Let $h : U \to \cc$ be a $C^1$-function with $\sup_{z \in U}\|(dh)(z)\|_\infty < \infty$. It is a basic consequence of the mean value theorem that
\[
|h(x) - h(y)| \leq \|\phi(x) - \phi(y)\|_{\rr^N} \cd
\sup_{z \in U} \|(dh)(z)\|_\infty \cd \ga_\phi,
\]
for all $x,y \in U$, see \cite[Chapter I, Corollary 4.2]{Lan:IDM}.

Let now $\xi \in H$, $\eta \in G$. It follows from the above estimate that
\[
\begin{split}
& |\inn{\eta,f(\xi)}(x) - \inn{\eta,f(\xi)}(y)|
\leq \|\phi(x) - \phi(y)\|_{\rr^N} \cd 
\sup_{z \in U} \|d\inn{\eta,f(\xi)}(z)\|_\infty \cd \ga_\phi \\
& \q \leq \|\phi(x) - \phi(y)\|_{\rr^N} \cd \|\eta\|_G \cd \|\xi\|_H
\cd \sup_{z \in U}\|(df)(z)\|_\infty \cd \ga_\phi.
\end{split}
\]
These inequalities prove the lemma.
\end{proof}

\begin{lemma}\label{l:adjest}
Let $f : \C M \to \sL(H,G)$ be $*$-strongly differentiable. Then
$\|d(f^*)(x)\|_\infty \leq \sqrt{N} \cd \|df(x)\|_\infty$ for each $x \in \C M$, where $N \in \nn$ is the dimension of the smooth Riemannian manifold $\C M$.
\end{lemma}
\begin{proof}
Let $x \in \C M$. Choose a smooth chart $\phi : U \to \rr^N$ with $x \in U$ such that $\{(d\phi_i)(x)\}_{i\in \{1,\ldots,N\}}$ is an orthonomal basis for $T_x^*\C M$.

For each $i \in \{1,\ldots,N\}$, let $A_i := \frac{\pa f}{\pa \phi_i}\big|_x : H \to G$ denote the strong derivative at $x \in \C M$. Remark that the adjoint is given by $A_i^* = \frac{\pa (f^*)}{\pa \phi_i}\big|_x : G \to H$.

Let $\xi \in H$ and use that $\{(d\phi_i)(x)\}$ is an orthonormal basis for $T_x^*\C M$ to compute as follows,
\[
\|d(f)(x)(\xi)\|_{G \ot T_x^*\C M}^2 =  \sum_{i=1}^N \inn{A_i \xi,A_i \xi}
= \inn{\xi,\sum_{i=1}^N A_i^* A_i \xi}.
\]
This shows that $\|d(f)(x)\|_\infty = \|\sum_{i=1}^N A_i^* A_i\|_\infty^{1/2}$. Similarly, $\|d(f^*)(x)\|_\infty = \|\sum_{i=1}^N A_i A_i^*\|_\infty^{1/2}$.

Since it follows from basic $C^*$-algebra theory that $\|\sum_{i=1}^N A_i A_i^*\|_\infty \leq N \cd \|\sum_{i=1}^N A_i^* A_i\|_\infty$, this proves the lemma.
%
%
\end{proof}

From now on it will be assumed that the smooth Riemannian manifold $\C M$ has bounded geometry. 


\begin{dfn}\label{d:hilbou}
Let $r,s \in (0,r_{\T{inf}})$ satisfy the bounded geometry conditions of Definition \ref{d:bougeo} and apply the notation $(\phi_x,U_x) := (\phi_{x,r},U_{x,r})$ and $(\phi_x,V_x) := (\phi_{x,s},U_{x,s})$ for the normal charts at a point $x \in \C M$.

A \emph{Hilbert bundle of bounded geometry} over $\C M$ consists of a topological Hausdorff space $\sH$, a continuous surjective map $\pi : \sH \to \C M$, and a fiber preserving homeomorphism $\psi_x : \pi^{-1}(U_x) \to U_x \ti H$ for each open set in the cover $\{U_x\}_{x \in \C M}$ of $\C M$, such that
\begin{enumerate}
\item The transition maps $\tau_{x,y} := \psi_x \psi_y^{-1} : U_x \cap U_y \to \sL(H)$ are strongly differentiable with $\tau_{i,j}(z)$ unitary for each $z \in U_x \cap U_y$.
\item $\T{sup}_{x,y \in \C M} \|d\tau_{x,y}\|_\infty < \infty$.
\end{enumerate}

Let $\sG$ be another Hilbert bundle of bounded geometry with local trivializations $\rho_x : \pi^{-1}(V_x) \to V_x \ti G$. A \emph{morphism} of Hilbert bundles of bounded geometry is a continuous fiber preserving map $\al : \sH \to \sG$ with $*$-strongly differentiable transition maps $\al_{y,x} := \rho_y  \al \psi_x^{-1} : V_y \cap U_x \to \sL(H,G)$ and with
\[
\T{sup}_{x,y \in \C M}\big( \|d\al_{y,x}\|_\infty + \|\al_{y,x}\|_\infty\big)< \infty.
\]

The \emph{category of Hilbert bundles of bounded geometry} over $\C M$ is denoted by $\G{Hilb}_{\C M}$.
\end{dfn}

Recall from above that $(d\al_{y,x})(z) : H \to G \ot T_z^* \C M$ is defined by $(d\al_{y,x})(z)(\xi) = d( \al_{y,x}(\xi))(z)$, where $\al_{y,x}(\xi) : V_y \cap U_x \to G$ for each $\xi \in H$. The norms appearing in the above definition refer to supremum norms, as an example,
\[
\|d \al_{y,x}\|_\infty
:= \T{sup}_{z \in V_y \cap U_x} \|(d \al_{y,x})(z)\|_\infty.
\]
%

\begin{remark}\label{r:idemor}
The identity map $1_{\sH} : \sH \to \sH$ is a morphism of Hilbert bundles of bounded geometry. Indeed, this follows immediately from the conditions $(1)$ and $(2)$ of Definition \ref{d:hilbou}. Notice that $\tau_{x,y}^* = \tau_{y,x}$ for all $x,y \in \C M$.
\end{remark}

\begin{remark}\label{r:her}
Each fiber $\sH_z := \pi^{-1}(\{z\})$ of a Hilbert bundle of bounded geometry is a Hilbert space with inner product $\inn{\cd,\cd}_z : \sH_z \ti \sH_z \to \cc$, $\inn{\xi,\eta}_z := \inn{\psi_x(\xi),\psi_x(\eta)}$ for any $x \in \C M$ with $z \in U_x$. Notice that this inner product is well-defined since $\tau_{x,y}(z)$ is unitary when $z \in U_x \cap U_y$.

Let $\al : \sH \to \sG$ be a morphism of Hilbert bundles of bounded geometry. Then the \emph{adjoint morphism} $\al^* : \sG \to \sH$ is defined fiber wise by $(\al^*)_z := (\al_z)^* : \sG_z \to \sH_z$ using the above Hilbert space structures.

The adjoint morphism is a morphism of Hilbert bundles of bounded geometry since $(\al^*)_{x,y}(z) = \al_{y,x}(z)^*$ for each $x,y \in \C M$ and each $z \in U_x \cap V_y$. This implies that the transition maps for the adjoint morphism are $*$-strongly differentiable with
\[
\T{sup}_{x,y \in \C M} \big( \|d(\al^*_{x,y})\|_\infty + \|\al^*_{x,y}\|_\infty\big)
\leq 
\sqrt{N} \cd \T{sup}_{x,y \in \C M}\big( \|d(\al_{y,x})\|_\infty 
+ \|\al_{y,x}\|_\infty \big)
< \infty,
\]
by Lemma \ref{l:adjest}.

The morphism $\al$ is \emph{unitary} when $\al^* \ci \al = 1_{\sH}$ and $\al \ci \al^* = 1_{\sG}$.
\end{remark}



\begin{dfn}\label{d:HilBunII}
A \emph{section} of a Hilbert bundle of bounded geometry $\pi : \sH \to \C M$ is a continuous map $s : \C M \to \sH$ such that $\pi \ci s = 1_{\C M}$. The section is \emph{continuously differentiable} (or $C^1$) when $\psi_x \ci s : U_x \to H$ is continuously differentiable for each $x \in \C M$.

A $C^1$-section \emph{vanishes at infinity} when there for each $\ep>0$ exists a compact set $\C K \su \C M$ such that
\[
\T{sup}_{x \in \C M} \Big( \T{sup}_{y \in (\C M\setminus \C K)\cap U_x}\big( \|(\psi_x \ci s)(y)\|_H + \|d(\psi_x \ci s)(y)\|_{H \ot T_y^*\C M} \big)\Big) < \ep.
\]

The set of $C^1$-sections of $\sH$ which vanish at infinity is denoted by $\Ga_0^1(\sH)$. It is a right module over the ring $C^1_0(\C M)$. 

A $C^1$-section which vanish at infinity will often be referred to as a \emph{$C^1_0$-section}.
\end{dfn}

Let $\Ga_0^1 : \G{Hilb}_{\C M} \to \G{Mod}_{C^1_0(\C M)}$ denote the covariant functor which maps a Hilbert bundle of bounded geometry $\sH \to \C M$ to the right module of $C^1_0$-sections $\Ga^1_0(\sH)$, and a morphism of Hilbert bundles of bounded geometry $\al : \sH \to \sG$ to the morphism of modules, $\Ga^1_0(\al) : \Ga^1_0(\sH) \to \Ga^1_0(\sG)$, $\Ga^1_0(\al)(s) := \al \ci s$. 

Remark that the derivative of $\rho_x \ci \al \ci s : U_x \to G$ is given by
\[
d(\al_{x,x} \ci \psi_x \ci s)(y) = d(\al_{x,x})(y)(\psi_x \ci s)(y) + (\al_{x,x}(y) \ot 1)d(\psi_x \ci s)(y),
\]
for all $y \in U_x$. It follows that the $C^1$-section $\Ga_0^1(\al)(s) : \C M \to \sG$ vanishes at infinity.

\section{Operator $*$-modules}\label{s:opemodI}

\subsection{Operator $*$-algebras}

Let $H$ and $G$ be Hilbert spaces, and let $X \su \sL(H,G)$ be a subspace which is closed in the operator norm. Then the vector space $M(X) := \lim_{n \to \infty} M_n(X)$ of finite matrices over $X$ has a canonical norm $\|\cd\|_X$ coming from the identifications $M_n(X) \su M_n(\sL(H,G)) \cong \sL(H^n,G^n)$. The properties of the pair $\big( M(X), \|\cd\|_X\big)$ are crystallized in the next definition.

Notice that the above construction yields a canonical norm $\|\cd\|_\cc : M(\cc) \to [0,\infty)$ on the finite matrices over $\cc$ since $\cc \cong \sL(\cc,\cc)$. For each $n \in \nn$ the norm $\|\cd\|_\cc : M_n(\cc) \su M(\cc) \to [0,\infty)$ coincides with the unique $C^*$-algebra norm.

\begin{dfn}\label{d:opespa}
An \emph{operator space} is a vector space $X$ with a norm $\|\cd\|_X$ on the finite matrices $M(X) := \lim_{n \to \infty} M_n(X)$ such that
\begin{enumerate}
\item The normed space $X \su M(X)$ is a Banach space.
\item The inequality $\|v \cd \xi \cd w\|_X \leq \|v\|_{\cc} \cd \|\xi\|_X \cd \|w\|_\cc$ holds for all $v,w \in M(\cc)$ and all $\xi \in M(X)$.
\item The equality $\|\xi \op \eta \|_X = \T{max}\{\|\xi\|_X, \|\eta\|_X\}$ holds for all $\xi \in M_n(X)$ and $\eta \in M_m(X)$, where $\xi \op \eta \in M_{n+m}(X)$ is the direct sum of the matrices.
%
\end{enumerate}

A \emph{morphism} of operator spaces is a \emph{completely bounded} linear map $\al : X \to Y$. The term completely bounded means that $\al_n : M_n(X) \to M_n(Y)$ is a bounded operator for each $n \in \nn$ and that $\T{sup}_n \|\al_n\|_\infty < \infty$. The supremum is denoted by $\|\al\|_{\T{cb}} := \T{sup}_n\|\al_n\|_\infty$ and is referred to as the \emph{completely bounded} norm.
\end{dfn}

By a fundamental theorem of Ruan each operator space $X$ is completely isometric to a closed subspace of $\sL(H)$ for some Hilbert space $H$. See \cite[Theorem 3.1]{Rua:SCA}.

\begin{dfn}
An \emph{operator algebra} is an operator space $A$ with a \emph{completely bounded} product $m : A \ti A \to A$ which (together with the vector space structure) makes $A$ an algebra over $\cc$. The complete boundedness means that there exists a constant $C > 0$ such that $\|a \cd b\|_A \leq C \cd \|a\|_A \cd \|b\|_A$ for all finite matrices $a,b \in M(A)$.

A \emph{morphism} of operator algebras $\al : A \to B$ is a morphism of the underlying operator space structures with $\al(a \cd b) = \al(a) \cd \al(b)$ for all $a,b \in A$.
\end{dfn}

By a fundamental theorem of Blecher, Ruan, and Sinclair each operator algebra is isomorphic to a closed subalgebra of the bounded operators on some Hilbert space, see \cite[Theorem 3.1]{BlRuSi:COA} and \cite[Theorem 2.2]{Ble:COA}.

\begin{dfn}\label{d:op*alg}
An \emph{operator $*$-algebra} is an operator algebra $A$ with a completely bounded involution $\da : A \to A$ which (together with the $\cc$-algebra structure) makes $A$ a $*$-algebra. The involution is defined on finite matrices by $(a^\da)_{ij} := (a_{ji})^\da$. It is thus required that there exists a constant $C > 0$ such that $\|a^\da\|_A \leq C \cd \|a\|_A$ for all finite matrices $a \in M(A)$.
\end{dfn}

Let $A \su \sL(H)$ be a closed subalgebra. Then $A$ is an operator $*$-algebra when there exists an invertible selfadjoint operator $g \in \sL(H)$ such that $g a^* g^{-1} \in A$ for all $a \in A$. The involution is given by $a^\da := g a^* g^{-1}$. Note that the usual involution $*$ on the bounded operators does \emph{not} necessarily map $A$ into itself. 

It follows from the above that any $C^*$-algebra is an operator $*$-algebra. The concept of an operator $*$-algebra is however more flexible as the following example indicates.

%

\subsubsection{Example: Differentiable functions}\label{ss:diffun}
Let $\C M$ be a smooth Riemannian manifold of dimension $N \in \nn$.

Remark that a Hilbert space $H$ becomes an operator space when identified with the bounded operators $\sL(\cc,H)$. More explicitly the matrix norm is defined by
\[
\|\xi\|_H := \|\inn{\xi,\xi}\|_{\cc}^{1/2} \, \, , \, \, \xi \in M(H), 
\]
where $\inn{\xi,\xi}_{ij} := \sum_k \inn{\xi_{ki},\xi_{kj}}$. There is in particular an operator space structure on each fiber $T_x^*\C M$ of the cotangent bundle coming from the Riemannian metric.

Define the norm $\|\cd\|_1 : M(C^1_0(\C M)) \to [0,\infty)$ by
\[
\|f\|_1 := \T{sup}_{x \in \C M} \big( \|f(x)\|_\cc^2 + \|df(x)\|_{T_x^*\C M}^2 \big)^{1/2}
\, \, , \, \, f \in M(C^1_0(\C M)),
\]
where $d : M(C^1_0(\C M)) \to M(\Ga_0(T^*\C M))$ is the de Rham differential applied entrywise.

Recall that the complex conjugation on $C^1_0(\C M)$ extends to an involution $\da : M(C^1_0(\C M)) \to M(C^1_0(\C M))$ defined by $(f^\da)_{ij} := \ov{f_{ji}}$.

The next proposition is a variation of \cite[Proposition 2.8]{KaLe:SFU}.

\begin{prop}\label{p:diffun}
The $*$-algebra  $C^1_0(\C M)$ is an operator $*$-algebra when equipped with the matrix norm $\|\cd\|_1 : M(C^1_0(\C M)) \to [0,\infty)$. 
\end{prop}
\begin{proof}
Using that $\cc$ and the fibers $T_x^*\C M$ are operator spaces it is not hard to prove that $C^1_0(\C M)$ is an operator space. Indeed, the norm $\|\cd\|_1$ can be rewritten as $\|f\|_1 = \T{sup}_{x \in \C M}\big\|\big(f(x),df(x) \big)\big\|_{\cc \op T_x^* \C M}$, where $\cc \op T_x^*\C M$ is given the direct sum Hilbert space structure.
%

Let $f,g \in M(C^1_0(\C M))$ be finite matrices and let $x \in \C M$. The Leibnitz rule and the operator space structures on $\cc$ and the fiber $T_x^*\C M$ then implies that
\[
\begin{split}
& \|(f\cd g)(x)\|_{\cc}^2 +  \|d(f \cd g)(x)\|_{T_x^*\C M}^2 \\
& \q \leq \|f(x)\|_\cc^2 \cd \|g(x)\|_\cc^2 + \big( \|(df)(x)\|_{T_x^*\C M} \cd \|g(x)\|_\cc + \|f(x)\|_\cc \cd \|(dg)(x)\|_{T_x^*\C M} \big)^2 \\
& \q \leq 5 \cd \|f\|_1^2 \cd \|g\|_1^2,
\end{split}
\]
for all $x \in \C M$. This shows that $C^1_0(\C M)$ is an operator algebra.

Let $f \in M(C^1_0(\C M))$ be a finite matrix and let $x \in \C M$. Without loss of generality it may be supposed that $f \in M_n(C^1_0(\C M))$ for some $n \in \nn$. The element $f$ may thus be perceived as a $*$-strongly differentiable map $\ka(f) : \C M \to \sL(\cc^n,\cc^n)$. 

Notice now that $\ka(f^\da) = \ka(f)^*$ and furthermore that $\|f(x)\|_\cc = \|\ka(f)(x)\|_\infty$ and $\|(df)(x)\|_{T_x^*\C M} = \|d(\ka(f))(x)\|_\infty$. The result of Lemma \ref{l:adjest} then implies that
\[
\begin{split}
& \|f(x)\|_\cc^2 + \|(df)(x)\|^2_{T_x^*\C M}
= \|\ka(f)(x)\|_\infty^2 + \|d(\ka(f))(x)\|_\infty^2 \\
& \q \leq  N \cd \big( \|\ka(f)^* (x)\|_\infty^2 + \|d(\ka(f)^*)(x)\|^2_\infty \big)
= N \cd \big( \|f^\da (x)\|_\cc^2 +  \|(df^\da)(x)\|^2_{T_x^*\C M} \big).
\end{split}
\]
This proves that the involution $\da : C^1_0(\C M) \to C^1_0(\C M)$ given by complex conjutation is completely bounded.
\end{proof}

\subsection{Hermitian operator modules}\label{s:herope}

\begin{dfn}
Let $A$ be an operator algebra. A (right) \emph{operator module} over $A$ is an operator space $X$ with a completely bounded right module action of $A$. Thus, there exists a constant $C > 0$ such that
\[
\|\xi \cd a\|_X \leq C \cd \|\xi\|_X \cd \|a\|_A
\]
for all finite matrices $\xi \in M(X)$ and $a \in M(A)$. A \emph{morphism} of operator modules over $A$ is a morphism $\al : X \to Y$ of the underlying operator spaces with $\al(\xi \cd a) = \al(\xi) \cd a$.
%
%
\end{dfn}

Any operator module $X$ over $A$ is isomorphic to a concrete operator module. This means that there exists a Hilbert space $H$, a closed subspace $Y \su \sL(H)$ and a closed subalgebra $B \su \sL(H)$ together with isomorphisms $\phi : X \cong Y$ and $\pi : A \cong B$ such that $\phi(\xi \cd a) = \phi(\xi) \ci \pi(a)$ for all $\xi \in X$ and $a \in A$. See \cite[Theorem 2.2]{Ble:GHM} and \cite[Corollary 3.3]{CES:CMC}.

\begin{dfn}\label{d:hermod}
Let $A$ be an operator $*$-algebra. A \emph{hermitian operator module} over $A$ is an operator module $X$ over $A$ with a completely bounded pairing $\inn{\cd,\cd} : X \ti X \to A$ such that
\begin{enumerate}
\item $\inn{\xi,\eta \cd a} = \inn{\xi,\eta} \cd a$ for all $\xi,\eta \in X$, and all $a \in A$.
\item $\inn{\xi,\eta \cd \la + \ze \cd \mu} = \inn{\xi,\eta}\cd \la + \inn{\xi,\ze} \cd \mu$ for all $\xi,\eta,\ze \in X$, and all $\la,\mu \in \cc$.
\item $\inn{\xi,\eta} = \inn{\eta,\xi}^\da$ for all $\xi,\eta \in X$.
\end{enumerate}
The complete boundedness means that there exists a constant $C > 0$ such that
\[
\|\inn{\xi,\eta}\|_A \leq C \cd \|\xi\|_X \cd \|\eta\|_X
\]
for all finite matrices $\xi,\eta \in M(X)$, where $\inn{\xi,\eta}_{ij} := \sum_k \inn{\xi_{ki},\eta_{kj}}$.

The map $\inn{\cd,\cd} : X \ti X \to A$ will be referred to as a \emph{hermitian form}.
%

A \emph{morphism} of hermitian operator modules is a morphism of operator modules $\al : X \to Y$ for which there exists a morphism of right operator modules $\al^* : Y \to X$ such that $\inn{\al(\xi),\eta}_Y = \inn{\xi,\al^*(\eta)}_X$ for all $\xi \in X$ and $\eta \in Y$. The morphism $\al^* : Y \to X$ is referred to as an \emph{adjoint} of $\al$. The morphism $\al$ is \emph{unitary} when $\al$ is invertible with $\al^* = \al^{-1}$.

The hermitian form $\inn{\cd,\cd}$ is \emph{non-degenerate} when the implication
\[
\big( \inn{\xi,\eta} = 0 \, \, \forall \xi \in X \big) \Rightarrow \eta = 0
\]
holds for all $\eta \in X$.
%
\end{dfn}

Remark that the adjoint $\al^* : Y \to X$ of a morphism $\al : X \to Y$ of non-degenerate hermitian operator modules is unique.

A basic example of a hermitian operator module is an operator $*$-algebra when considered as a right module over itself and equipped with the hermitian form $\inn{\cd,\cd} : A \ti A \to A$ given by $\inn{a,b} := a^\da \cd b$. Remark that $\inn{\cd,\cd}$ is completely bounded since $\inn{a,b} = a^\da \cd b$ for all finite matrices $a,b \in M(A)$. See Definition \ref{d:op*alg}.

In the following sections we will see more examples of hermitian operator modules.

\subsection{The standard module}\label{s:stamod}

Let $X$ be an operator space. Let $c_c(X)$ denote the vector subspace of $M(X)$ such that $\xi \in c_c(X) \lrar ( \xi_{ij} = 0 \T{ when } j > 1)$. The following notation will be in effect, $\xi(k) := \xi_{k1}$ for each $\xi \in c_c(X)$.
%

For each $n \in \nn$, define the map $\al_n : M(c_c(X)) \to M(X)$ by 
\[
\al_n(\xi)_{(i-1)n + k,j} := \xi_{i,j}(k) \q \xi \in M(c_c(X)),
\]
for all $i,j \in \nn$ and all $k \in \{1,\ldots,n\}$. In other words, $\al_n(\xi)$ is the block matrix where each block is an $(n \ti 1)$-matrix and the block in position $(i,j)$ is $P_n(\xi_{ij}) := (\xi_{ij}(1),\ldots,\xi_{ij}(n))^t$.

For each $n \in \nn$, notice that
\begin{equation}\label{eq:stamodI}
\al_n(v \cd \xi \cd w) = \be_n(v) \cd \al_n(\xi) \cd w,
\end{equation}
for all $v,w \in M(\cc)$ and $\xi \in M(c_c(X))$, where $\be_n(v)$ is the block matrix where each block is an $(n \ti n)$-matrix and the block in position $(i,j)$ is the diagonal matrix $\T{diag}(v_{ij})$. Thus, in formulas $\be_n(v)_{(i-1)n+k,(j-1)n+l} = \de_{kl} \cd v_{i,j}$, for all $i,j \in \nn$ and all $k,l \in \{1,\ldots,n\}$, where $\de_{kl} \in \{0,1\}$ is the Kronecker delta.

Define the matrix norm $\|\cd \|_{X^\infty} : M(c_c(X)) \to [0,\infty)$ by 
\[
\|\xi\|_{X^\infty} := \lim_{n \to \infty}\|\al_n(\xi)\|_X = \sup_{n \in \nn}\|\al_n(\xi)\|_X,
\]
for all finite matrices $\xi \in M(c_c(X))$. Using \eqref{eq:stamodI} it is not hard to see that this matrix norm satisfies the properties $(2)$ and $(3)$ in Definition \ref{d:opespa}.

\begin{dfn}
The \emph{standard sequence space} $X^\infty$ over $X$ is the operator space obtained as the completion of $c_c(X)$ with respect to the norm $\|\cd\|_{X^\infty}$ defined above.
\end{dfn}
%

\begin{prop}\label{p:infher}
Let $X$ be a hermitian operator module over an operator $*$-algebra $A$. Then $X^\infty$ becomes a hermitian operator module when equipped with the right action induced by $c_c(X) \ti A \to c_c(X)$, $(\xi \cd a)(i) := \xi(i) \cd a$ and the hermitian form induced by $\inn{\cd,\cd}_{X^\infty} : c_c(X) \ti c_c(X) \to A$, $\inn{\xi,\eta}_{X^\infty} := \sum_i \inn{\xi(i),\eta(i)}_X$.
\end{prop}
\begin{proof}
We will only prove that the right action and the hermitian form are completely bounded, leaving the rest of the verifications to the reader.

The right action is completely bounded since
\[
\al_n(\xi \cd a) = \al_n(\xi) \cd a
\]
for all $n \in \nn$, $\xi \in M(c_c(X))$, and $a \in M(A)$. Indeed, using that the right action on $X$ is completely bounded, this implies the existence of a constant $C > 0$ such that
\[
\|\xi \cd a\|_{X^\infty} = \T{sup}_{n \in \nn}\|\al_n(\xi) \cd a\|_X \leq C \cd \T{sup}_{n \in \nn}\|\al_n(\xi)\|_X \cd \|a\|_A = C \cd \|\xi\|_{X^\infty} \cd \|a\|_A,
\]
for all finite matrices $\xi \in M(c_c(X))$ and $a \in M(A)$.

To see that the hermitian form is completely bounded note that
\[
\binn{\xi,\eta}_{X^\infty} = \lim_{n \to \infty}\binn{\al_n(\xi), \al_n(\eta)}_X,
\]
for all $\xi,\eta \in M(c_c(X))$. Indeed, there exists an $n_0 \in \nn$ such that
\[
\begin{split}
\big(\binn{\xi,\eta}_{X^\infty}\big)_{ij} 
& = \sum_{k = 1}^\infty \inn{\xi_{ki},\eta_{kj}}_{X^\infty}
= \sum_{k=1}^\infty \sum_{l = 1}^n \inn{\xi_{ki}(l),\eta_{kj}(l)} \\
& = \sum_{k=1}^\infty \sum_{l = 1}^n \inn{\al_n(\xi)_{(k-1)n + l,i},\al_n(\eta)_{(k-1)n + l,j}}
= \inn{\al_n(\xi),\al_n(\eta)}_{ij}
\end{split}
\]
for all $n \geq n_0$ and all $i,j \in \nn$. Using the complete boundedness of the hermitian form $\inn{\cd,\cd}_X$, this implies that there exists a constant $C > 0$ such that
\[
\begin{split}
\|\inn{\xi,\eta}_{X^\infty}\|_A
& =  \lim_{n \to \infty}\big\|\binn{\al_n(\xi), \al_n(\eta)}_X\big\|_A
\leq C \cd \lim_{n \to \infty}\|\al_n(\xi)\|_X \cd \|\al_n(\eta)\|_X \\
& = C \cd \|\xi\|_{X^\infty} \cd \|\eta\|_{X^\infty},
\end{split}
\]
for all $\xi,\eta \in M(X^\infty)$.
\end{proof}

\begin{dfn}
Let $A$ be an operator $*$-algebra. The \emph{standard hermitian module} over $A$ is the standard sequence space $A^\infty$.
\end{dfn}

By the argument presented in the end of Section \ref{s:herope}, an operator $*$-algebra $A$ can be viewed as a hermitian operator module over itself. It  therefore follows from Proposition \ref{p:infher} that $A^\infty$ is a hermitian operator module. Explicitly, the completely bounded right action is induced by $(x \cd a)(i) := x(i) \cd a$ for all $x \in c_c(A)$ and $a \in A$. The completely bounded hermitian form is induced by $\inn{x,y}_{A^\infty} = \sum_i x(i)^\da \cd y(i)$ for all $x,y \in c_c(A)$.

%
%

\subsubsection{The standard module of a Riemannian manifold}\label{ss:starie}
Let $\C M$ be a smooth Riemannian manifold of dimension $N \in \nn$, and let $H$ be a separable Hilbert space of infinite dimension.

Consider the right module $C^1_0(\C M,H)$ over $C^1_0(\C M)$ of Hilbert space valued $C^1_0$-maps. See Definition \ref{d:HilBunI}. 

%

The next proposition is a variation of \cite[Proposition 3.6]{KaLe:SFU}.

\begin{prop}\label{p:starie}
The module $C^1_0(\C M,H)$ is isomorphic to the standard hermitian module $C_0^1(\C M)^\infty$.
\end{prop}
\begin{proof}
Remark first that $C^1_0(\C M,H)$ becomes a Banach space when equipped with the norm
\[
\|\cd\|_1 : s \mapsto \sup_{x \in \C M}\big( \inn{s(x),s(x)}_H + \inn{(ds)(x),(ds)(x)}_{H \ot T_x^*\C M} \big)^{1/2}.
\]

Let $\{e_k\}$ be an orthonormal basis for $H$. Define the map $\be : c_c\big( C_0^1(\C M) \big) \to C^1_0(\C M,H)$ by $\be(f)(x) := \sum_{n=1}^\infty e_k \cd f(k)(x)$, where $f(k) \in C^1_0(\C M)$ denotes the $k^{\T{th}}$ entry of the column $f \in c_c\big(C^1_0(\C M)\big)$, see the beginning of Section \ref{s:stamod}.

Notice then that
\[
\begin{split}
\|\be(f)\|_1
& =
\sup_{x \in \C M} \Big( \sum_{k=1}^\infty \binn{f(k)(x),f(k)(x)} +
\sum_{k=1}^\infty \binn{d(f(k))(x),d(f(k))(x)} \Big)^{1/2} \\
& = \sup_{x \in \C M}\Big( \|f(x)\|_{\cc}^2 + \|(df)(x)\|_{T_x^*\C M}^2 \Big)^{1/2}
= \|f\|_{C^1_0(\C M)^\infty},
\end{split}
\]
where $f \in c_c\big( C_0^1(\C M)\big)$ has been identified with the element $\lim_{n \to \infty} \al_n(f) \in M(C^1_0(\C M))$.

This implies that $\be$ induces an isometry
\[
\be : C_0^1(\C M)^\infty \to C^1_0(\C M,H).
\]
In particular, $\be$ is injective and has closed image. 

Let now $s \in C^1_0(\C M,H)$. It is then not hard to see that $s = \lim_{m \to \infty} s_m$ where $s_m(x) := \sum_{i=1}^m e_i \cd \inn{e_i,s(x)}$ for all $m \in \nn$. It follows that $\be$ has dense image. This shows that $\be$ is an isometric isomorphism of Banach spaces. Since $\be$ also respects the module structures over $C^1_0(\C M)$ this proves the proposition.
%
\end{proof}

\begin{remark}
It follows from Proposition \ref{p:starie}, Proposition \ref{p:infher}, and Proposition \ref{p:diffun} that $C^1_0(\C M,H)$ is a hermitian operator module over $C^1_0(\C M)$. The matrix norm $\|\cd\|_1 : M\big( C^1_0(\C M,H) \big) \to [0,\infty)$ is defined by
\[
\|s\|_1 := \sup_{x \in \C M} \Big( \|s(x)\|^2_H + \|(ds)(x)\|_{H \ot T_x^*\C M}^2 \Big)^{1/2}
\]
for all finite matrices $s \in M(C^1_0(\C M))$. The completely bounded hermitian form $\inn{\cd,\cd} : C^1_0(\C M,H) \ti C^1_0(\C M,H) \to C^1_0(\C M)$ is given by $\inn{s,t}(x) := \inn{s(x),t(x)}_H$. Equipped with this structure $C^1_0(\C M,H)$ becomes unitarily completely isometric to the standard hermitian module $C^1_0(\C M)^\infty$.
\end{remark}

\subsection{Operator $*$-modules}\label{s:opemod}

Let $A$ be an operator $*$-algebra and let $P : A^\infty \to A^\infty$ be a completely bounded projection, thus $P^2 = P = P^*$, where the adjoint is with respect to the canonical hermitian form $\inn{\cd,\cd}_{A^\infty} : A^\infty \ti A^\infty \to A$. See Proposition \ref{p:infher}.

Let $X := PA^\infty := \{x \in A^\infty \, | \, x = P(x)\}$. Notice that $X \su A^\infty$ is a closed submodule over $A$ and that $\inn{\cd,\cd}_{A^\infty}$ restricts to a pairing $\inn{\cd,\cd}_X : X \ti X \to A$ which satisfies the algebraic conditions in Definition \ref{d:hermod}. It follows that $X$ becomes a hermitian operator module over $A$ when equipped with the matrix norm $\|\cd\|_X : x \mapsto \|x\|_{A^\infty}$ for all $x \in M(X)$.

The next definition is a reformulation of \cite[Definition 3.4]{KaLe:SFU}.

\begin{dfn}\label{d:opemod}
An \emph{operator $*$-module} is a hermitian operator module $X$ over an operator $*$-algebra $A$ which is unitarily isomorphic to a direct summand $PA^\infty$ of the standard module $A^\infty$.

A morphism of operator $*$-modules $\al : X \to Y$ is a morphism of the underlying hermitian operator modules.

The category of operator $*$-modules over $A$ is denoted by $\G{Op^*Mod}_A$.
\end{dfn}

%

\subsubsection{Morphisms of the standard hermitian module over a complete manifold}\label{ss:comman}
Let $H$ and $G$ be separable Hilbert spaces of infinite dimension and let $\C M$ be a smooth Riemannian manifold of dimension $N \in \nn$.

Let $C^1_b(\C M,\sL(H,G))$ denote the vector space of $*$-strongly differentiable maps $\al : \C M \to \sL(H,G)$ with
\[
\|\al\|_1 := \sup_{x \in \C M} \| \pi(\al)(x) \|_\infty < \infty,
\]
where $\pi(\al)(x)$ is the bounded operator defined by
\[
\pi(\al)(x) := \ma{cc}{
\al(x) & 0 \\
(d\al)(x) & \al(x) \ot 1
}  : H \op (H \ot T_x^*\C M) \to G \op (G \ot T_x^*\C M),
\]
for all $x \in \C M$. The vector space $C^1_b(\C M,\sL(H,G))$ becomes a Banach space when equipped with the norm $\|\cd\|_1$. Notice that $\al^* \in C^1_b(\C M,\sL(G,H))$ for all $\al \in C^1_b(\C M,\sL(H,G))$ by an application of Lemma \ref{l:adjest}.
%

Define the map
\[
\Phi : C^1_b(\C M,\sL(H,G)) \to \T{Mor}\big( C^1_0(\C M,H),C^1_0(\C M,G) \big)
\]
by $\Phi(\al)(s)(x) := \al(x)(s(x))$ for $\al \in C^1_b(\C M,\sL(H,G))$, $s \in C^1_0(\C M,H)$, and $x \in \C M$. Remark that the morphisms on the right hand side are the morphisms between the objects $C^1_0(\C M,H)$ and $C^1_0(\C M,G)$ in the category of operator $*$-modules. See Definition \ref{d:opemod} and Proposition \ref{p:starie}.
%
%
%
%

\begin{prop}\label{p:commanI}
Suppose that $\C M$ is complete. Then the above map $\Phi$ is well-defined and bijective. Furthermore, $\Phi(\al)^* = \Phi(\al^*)$ and $\|\Phi(\al)\|_{\T{cb}} \leq \|\al\|_1 \leq 2 \cd \|\Phi(\al)\|_{\T{cb}}$ for all $\al \in C^1_b(\C M,\sL(H,G))$.
%
%
\end{prop}
\begin{proof}
Let $\al : \C M \to \sL(H,G)$ be $*$-strongly differentiable with $\|\al\|_1 < \infty$. It is not hard to see that $\Phi(\al) : C^1_0(\C M,H) \to C^1_0(\C M,G)$ is a well-defined morphism of modules over $C^1_0(\C M)$. Furthermore, $\binn{\Phi(\al)(s),t} = \binn{s, \Phi(\al^*)(t)}$ for all $s \in C^1_0(\C M,H)$ and $t \in C^1_0(\C M,G)$.
%
%
%

To show that $\Phi(\al)$ is completely bounded, let $s \in C^1_0(\C M,H)$ and let $x \in \C M$. Remark that
\begin{equation}\label{eq:commanI1}
\begin{split}
\ma{c}{
\Phi(\al)(s)(x) \\
(d\Phi(\al)(s))(x)
} 
& = \ma{c}{
\al(x)(s(x)) \\
(d\al)(x)(s(x)) + (\al(x) \ot 1)(ds)(x)
} \\
& = \pi(\al)(x) \ma{c}{
s(x) \\
(ds)(x)
}.
\end{split}
\end{equation}

Recall that the Hilbert space direct sums $H \op H \ot T_x^*\C M$ and $G \op G \ot T_x^*\C M$ have canonical operator space structures. See the beginning of Section \ref{ss:diffun}. Furthermore, for any bounded operator $T : H \op (H \ot T_x^*\C M) \to G \op (G \ot T_x^*\C M)$ the completely bounded norm coincides with the operator norm, thus $\|T\|_\infty = \|T\|_{\T{cb}}$. See the discussion in \cite[Page 258]{Ble:AHM}.

It is thus a consequence of \eqref{eq:commanI1} that,
\[
\begin{split}
& \|\Phi(\al)(s)\|_1 
= \sup_{x \in \C M}\big(  
\|\Phi(\al)(s)(x)\|_G^2 + \|d\Phi(\al)(s)(x)\|_{G \ot T_x^*\C M}^2
\big)^{1/2} \\
& \q = \sup_{x \in \C M}\big\| \ma{c}{
\Phi(\al)(s)(x) \\
d\Phi(\al)(s)(x)
}\big\|_{G \op (G \ot T_x^*\C M)} \\
& \q \leq \sup_{x \in \C M} \Big( \|\pi(\al)(x)\|_\infty \cd \big\|\ma{c}{
s(x) \\
(ds)(x)}
 \big\|_{H \op (H \ot T_x^*\C M)} \Big)
\leq \|\al\|_1 \cd \|s\|_1
\end{split}
\]
for any \emph{finite matrix} $s \in M(C^1_0(\C M,H))$. It follows $\Phi(\al)$ is completely bounded with $\|\Phi(\al)\|_{\T{cb}} \leq \|\al\|_1$.

Since $\al^* \in C^1_b(\C M,\sL(G,H))$ the above argument implies that $\Phi(\al^*) = \Phi(\al)^*$ is completely bounded as well.

It is thus established that $\Phi : C^1_b(\C M,\sL(H,G)) \to \T{Mor}(C^1_0(\C M,H),C^1_0(\C M,G))$ is well-defined with $\Phi(\al^*) = \Phi(\al)^*$ and $\|\Phi(\al)\|_{\T{cb}} \leq \|\al\|_1$ for all $\al \in C^1_b(\C M,\sL(H,G))$.

Let now $\be : C^1_0(\C M,H) \to C^1_0(\C M,G)$ be a morphism of hermitian operator modules. Let $\Phi^{-1}(\be) : \C M \to \sL(H,G)$ be the unique map such that
\[
\Phi^{-1}(\be)(x)(\xi) = \be(s)(x)
\]
for all $x \in \C M$ $\xi \in H$ and $s \in C^1_0(\C M,H)$ with $s(x) = \xi$. It not hard to see that $\Phi^{-1}(\be)$ is $*$-strongly differentiable with $\Phi^{-1}(\be)^* = \Phi^{-1}(\be^*)$.
%

The completeness of $\C M$ implies the existence of a sequence $\{\si_i\}$ of smooth compactly supported functions such that
\begin{enumerate}
\item The image of $\si_i$ is contained in $[0,1]$ for all $i \in \nn$.
\item $\|d\si_i\|_\infty \leq 1/i$ for all $i \in \nn$.
\item For each compact set $\C K \su \C M$ there exists an index $i_0 \in \nn$ such that $\si_i(x) = 1$ for all $x \in \C K$ and all $i \geq i_0$.
\end{enumerate}
See for example \cite[Section 5]{Wol:ESD}, \cite[Page 117]{LaMi:SG}, \cite[Lemma 3.2.4]{Les:OCA}.

For each $\xi \in H$ and each $i \in \nn$, let $\xi \cd \si_i \in C_0^1(\C M,H)$ be defined by $(\xi \cd \si_i)(x) = \xi \cd \si_i(x)$ for all $x \in \C M$.

Let $\xi \in H$, $i \in \nn$ be fixed. Remark that
\[
\|\xi \cd \si_i\|_1
= \sup_{x \in \C M}\big( \|\xi \cd \si_i(x)\|_H^2 + \|\xi \ot (d\si_i)(x)\|^2_{H \ot T_x^*\C M}\big)^{1/2}
\leq \|\xi\|_H \cd (1 + 1/i^2)^{1/2}.
\]
Furthermore, it follows from \eqref{eq:commanI1} and the definition of $\Phi^{-1}(\be)$ that
\[
\pi(\Phi^{-1}(\be))(x)\ma{c}{ (\xi \cd \si_i)(x) \\
d(\xi \cd \si_i)(x)}
= \ma{c}{\be(\xi \cd \si_i)(x) \\ d\be(\xi \cd \si_i)(x) },
\]
for all $x \in \C M$. This entails that,
\begin{equation}\label{eq:commanI2}
\pi(\Phi^{-1}(\be))(x)\ma{c}{ (\xi \cd \si_i)(x) \\
d(\xi \cd \si_i)(x)}
\leq \|\be\|_{\T{cb}} \cd \|\xi \cd \si_i\|_1 \leq \|\xi\|_H \cd (1 + 1/i^2)^{1/2},
\end{equation}
for all $x \in \C M$.

As a consequence of \eqref{eq:commanI2} and the properties of the sequence $\{\si_i\}$ we get that,
\[
\pi(\Phi^{-1}(\be))(x)\ma{c}{\xi \\
0} \leq \|\be\|_{\T{cb}} \cd \|\xi\|_H,
\]
for all $x \in \C M$ and all $\xi \in H$. But this entails that 
\[
\|\Phi^{-1}(\be)\|_1 = \sup_{x \in \C M}\|\pi(\Phi^{-1}(\be))(x)\|_\infty \leq 2 \cd \|\be\|_{\T{cb}}.
\]
It is thus established that $\Phi^{-1} : \T{Mor}(C^1_0(\C M,H),C^1_0(\C M,G)) \to C^1_b(\C M,\sL(H,G))$ is well-defined with $\|\Phi^{-1}(\be)\|_1 \leq 2 \cd \|\be\|_{\T{cb}}$ for all morphisms $\be : C^1_0(\C M,H) \to C^1_0(\C M,G)$.

Since $(\Phi \ci \Phi^{-1})(\be) = \be$ and $(\Phi^{-1} \ci \Phi)(\al) = \al$, for all morphisms $\be : C^1_0(\C M,H) \to C^1_0(\C M,G)$ and all $*$-strongly differentiable maps $\al : \C M \to \sL(H,G)$ with $\|\al\|_1 < \infty$, this ends the proof of the proposition.
\end{proof}

\begin{remark}
It follows from the above proof that $\Phi$ is well-defined even when $\C M$ is not complete. Furthermore, $\Phi(\al)^* = \Phi(\al^*)$ and $\|\Phi(\al)\|_{\T{cb}} \leq \|\al\|_1$ for all $\al \in C^1_b(\C M,\sL(H,G))$. It is however unlikely that the inverse of $\Phi$ exists in such a general context due to the lack of a suitable approximate identity for $C^1_0(\C M)$.
\end{remark}

\section{Differentiable stabilization}\label{s:difstab}

Throughout this section $\pi : \sH \to \C M$ will be a Hilbert bundle of bounded geometry over the manifold $\C M$ of bounded geometry. The parameter $r \in (0,r_{\T{inf}})$ which satisfies the bounded geometry condition of Definition \ref{d:bougeo} will be fixed. Thus, for each $x \in \C M$, the notation $\psi_x : \pi^{-1}(U_x) \to U_x \ti H$ refers to the local trivialization over the normal chart $\phi_x : U_x \to \rr^N$. The model fiber $H$ is a separable Hilbert space (not necessarily of infinite dimension). See Definition \ref{d:hilbou}.

\begin{dfn}\label{d:boupar}
A \emph{bounded partition of unity} relative to the cover $\{U_x\}_{x \in \C M}$ is a countable partition of unity $\{\chi_i\}$ on $\C M$ such that
\begin{enumerate}
\item The squareroot $\sqrt{\chi_i}$ is compactly supported and smooth for each $i \in \nn$.
\item For each $i \in \nn$, there exists a point $x_i \in \C M$ such that $\T{supp}(\chi_i) \su U_{x_i}$.
\item The cover $\{\T{supp}(\chi_i)\}_{i \in \nn}$ of $\C M$ has finite multiplicity, where $\T{supp}(\chi_i) \su \C M$ denotes the support of $\chi_i$.
\item $C_\chi := \sup_{i \in \nn}\|d\sqrt{\chi_i}\|_\infty < \infty$.
\end{enumerate} 
\end{dfn}

A bounded partition of unity relative to $\{U_x\}_{x \in \C M}$ exists by Lemma \ref{l:finmul} and Lemma \ref{l:paruni}. We will often refer to the multiplicity of the cover $\{\T{supp}(\chi_i)\}$ as the multiplicity of the partition of unity $\{\chi_i\}$.
%

The notation $H^\infty:= \op_{i=1}^\infty H$ will refer to the infinite direct sum of the Hilbert space $H$ with itself. Thus $H^\infty$ is again a Hilbert space with inner product
\[
\inn{\sum_{i=1}^\infty e_i \cd \xi_i,\sum_{j=1}^\infty e_j \cd \eta_j}_{H^\infty} := \sum_{i=1}^\infty \inn{\xi_i,\eta_i},
\]
where $e_i \cd \xi_i \in H^\infty$ refers to the vector which has the vector $\xi_i \in H$ in position $i \in \nn$ and zeroes elsewhere.

%


Let $\{\chi_i\}$ be a bounded partition of unity relative to $\{U_x\}$ and let $\{x_i\}$ be a sequence of points with $\T{supp}(\chi_i) \su U_{x_i}$. Define the map
\[
\Phi : \Ga_0^1(\sH) \to C_0^1(\C M,H^\infty) \q
\Phi : s \mapsto \sum_{i=1}^\infty e_i \big( \psi_{x_i} \ci (s \cd \sqrt{\chi_i}) \big),
\]
where $e_i \big( \psi_{x_i} \ci (s \cd \sqrt{\chi_i}) \big)$ is the vector in $C_0^1(\C M,H^\infty)$ defined by 
\[
e_i \big( \psi_{x_i} \ci (s \cd \sqrt{\chi_i}) \big) : x \mapsto e_i \big( \psi_{x_i} \ci (s(x) \cd \sqrt{\chi_i}(x)) \q \forall x \in \C M.
\]

Recall from Proposition \ref{p:starie} that $C^1_0(\C M,H^\infty)$ is a hermitian operator module. The norm on $C^1_0(\C M,H^\infty)$ is given explicitly by
\[
\|\cd\|_1 : \sum_{i=1}^\infty e_i t_i \mapsto 
\sup_{x \in \C M}\big( \sum_{i=1}^\infty\|t_i(x)\|_H^2 + \sum_{i=1}^\infty \|(dt_i)(x)\|^2_{H \ot T_x^*\C M} \big).
\]

Recall also that $\Ga_0^1(\sH)$ denotes the module over $C^1_0(\C M)$ consisting of $C^1_0$-sections of $\sH$. See Definition \ref{d:HilBunII}.

\begin{prop}\label{p:isomet}
The above map $\Phi$ is a well-defined morphism of modules over $C^1_0(\C M)$.
%
\end{prop}
\begin{proof}
Let $s \in \Ga_0^1(\sH)$ and let $\ep > 0$. Since $s$ vanishes at infinity there exists a compact set $\C K \su \C M$ such that
\[
\|s(x)\|_{\sH_x} + \|d(\psi_{x_i} \ci s)(x)\|_{H \ot T_x^*\C M} < \ep
\]
for all $x \in \C M\setminus \C K$ and all $x_i \in \C M$ with $x \in U_{x_i}$.

To ease the notation, let $s(i) := \psi_{x_i} \ci s : U_{x_i} \to H$ for each $i \in \nn$. 

Let $x \in \C M \setminus \C K$. Estimate as follows,
\[
\sum_{i=1}^\infty \|(s(i) \cd \sqrt{\chi_i})(x)\|_H^2
= \sum_{i=1}^{\infty} \|s(x)\|_{\sH_x}^2 \cd \chi_i(x)
\leq \ep^2 \sum_{i=1}^\infty \chi_i(x) = \ep^2.
\]

Let $C_\chi := \T{sup}_{i \in \nn} \|d\sqrt{\chi_i}\| < \infty$ and let $K \in \nn$ be the multiplicity of the cover $\{\T{supp}(\chi_i)\}$.

Let $x \in \C M \sem \C K$. Compute as follows,
\begin{equation}\label{eq:difstabI}
\begin{split}
& \sum_{i=1}^{\infty}\|d(s(i) \cd \sqrt{\chi_i})(x)\|^2_{H \ot T_x^*\C M}
\leq K \cd \sup_{i \in \nn} \| d(s(i)\cd \sqrt{\chi_i} )(x)\|^2_{H \ot T_x^*\C M} \\
& \q \leq K \cd \sup_{i \in \nn} \big( \|d(s(i))(x) \cd \sqrt{\chi_i}(x)\|_{H \ot T_x^*\C M} + \|s(i)(x) \ot d\sqrt{\chi_i}(x)\|_{H \ot T_x^*\C M} \big)^2 \\
& \q \leq K \cd \ep^2 \cd (1 + C_\chi)^2.
\end{split}
\end{equation}

Choose an $i_0 \in \nn$ such that $\T{supp}(\chi_i) \su \C M\setminus \C K$ for all $i \geq i_0$. The two computations above then imply that
\[
\begin{split}
& \big\|\sum_{i = i_0}^{j_0} e_i (s(i) \cd \sqrt{\chi_i}) \big\|_1 \\
& \q = \sup_{x \in \C M \sem \C K}\big( \sum_{i=i_0}^{j_0} \|(s(i) \cd \sqrt{\chi_i})(x)\|_H^2
+ \sum_{i=i_0}^{j_0} \|d(s(i) \cd \sqrt{\chi_i})(x)\|_{H \ot T_x^*\C M}^2 \big)^{1/2} \\
& \q \leq \big( \ep^2 + K \cd \ep^2 \cd (1 + C_\chi)^2 \big)^{1/2}
%
\end{split}
\]
for all $j_0 \geq i_0$.

Since the constants $K > 0$ and $C_\chi > 0$ are independent of $\ep > 0$, this shows that the sequence $\{\sum_{i=1}^m e_i \cd (s(i) \cd \sqrt{\chi_i})\}_{m=1}^\infty$ is Cauchy in $C_0^1(\C M,H^\infty)$. The limit $\Phi(s) = \sum_{i=1}^\infty e_i \cd (s(i) \cd \sqrt{\chi_i})$ is therefore well-defined.

It is easily verified that $\Phi$ is a morphism of modules over $C^1_0(\C M)$.

This proves the proposition.
\end{proof}

Define the map 
\[
\Psi : C_0^1(\C M,H^\infty) \to \Ga_0^1(\sH) \q \Psi : \sum_{i=1}^\infty e_i \cd t_i \mapsto \sum_{i=1}^\infty \psi_{x_i}^{-1} \ci (t_i \cd \sqrt{\chi_i}),
\]
where the inclusion $\Ga_c^1(\sH|_{U_{x_i}}) \su \Ga_0^1(\sH)$ given by extension with zero has been suppressed. Remark that the sum in the definition of $\Psi$ makes sense at each point since the partition of unity $\{\chi_i\}$ is locally finite.

Recall that $\tau_{x,y} := \psi_x \ci \psi_y^{-1} : U_x \cap U_y \to \sL(H)$ denotes a transition function of the Hilbert bundle $\sH$ of bounded geometry. By definition these transition functions take values in the group of unitaries and they are strongly differentiable with
\[
C_\tau := \sup_{x,y \in \C M}\|d\tau_{x,y}\|_\infty < \infty.
\]

\begin{prop}\label{p:isosta}
The above map $\Psi$ is a well-defined morphism of modules over $C^1_0(\C M)$. 
%
%
\end{prop}
\begin{proof}
Let $t = \sum_{i=1}^\infty e_i t_i \in C_0^1(\C M,H^\infty)$. 

Let $x \in \C M$. For each $z \in U_x$ we have that
\[
\big( \psi_x \ci \Psi(t) \big)(z) = \sum_{i=1}^\infty \tau_{x,x_i}(z)(t_i \cd \sqrt{\chi_i})(z).
\]
Since the partition of unity $\{\chi_i\}$ is locally finite this shows that the section $\Psi(t) : \C M \to \sH$ is continuously differentiable.

To see that $\Psi(t)$ vanishes at infinity let $\ep > 0$. Since $t \in C_0^1(\C M,H^\infty)$ there exists a compact set $\C K \su \C M$ such that
\[
\sup_{x \in \C M \sem \C K}\Big( \sum_{i=1}^\infty \|t_i(x)\|_H^2 + 
\sum_{i=1}^\infty \|(dt_i)(x)\|_{H \ot T_x^*\C M} \Big)^{1/2} < \ep.
\]
In particular, we have that
\[
\sup_{x \in \C M}\|t_i(x)\|_H < \ep \, \, \T{ and } \, \, \sup_{x \in \C M}\|(dt_i)(x)\|_{H \ot T_x^*\C M} < \ep,
\]
for all $i \in \nn$.
%

Let $x \in \C M$ and suppose that $z \in (\C M \sem \C K) \cap U_x$. Compute as follows,
\begin{equation}\label{eq:difstabII}
\big\|\sum_{i=1}^{\infty} 
\tau_{x,x_i}(z)(t_i \cd \sqrt{\chi_i})(z) \big\|_H
\leq \sum_{i=1}^\infty \|(t_i \cd \sqrt{\chi_i})(x)\|_H
< \sum_{i=1}^\infty \ep \cd \sqrt{\chi_i}(x)
= \ep,
\end{equation}
where $K \in \nn$ denotes the multiplicity of the cover $\{\T{supp}(\chi_i)\}$. 

Let $C_\chi := \T{sup}_{i \in \nn} \|d\sqrt{\chi_i}\|_\infty < \infty$ and $C_\tau := \T{sup}_{v,w \in \C M}\|d(\tau_{v,w})\| < \infty$. Use the Leibnitz rule and the estimate in \eqref{eq:difstabI} to compute as follows,
\begin{equation}\label{eq:difstabIII}
\begin{split}
& \big\| \sum_{i=1}^\infty d\big(\tau_{x,x_i} (t_i \cd \sqrt{\chi_i})\big)(z) \big\|_{H \ot T_z^*\C M}
\leq
K \cd \sup_{i \in \nn} \big\| d\big(\tau_{x,x_i} (t_i \cd \sqrt{\chi_i})\big)(z) \big\|_{H \ot T_z^*\C M} \\
& \q \leq K \cd \sup_{i \in \nn}
\big\|d(\tau_{x,x_i})(z) (t_i \cd \sqrt{\chi_i})(z)
+ (\tau_{x,x_i}(z) \ot 1)d(t_i \cd \sqrt{\chi_i})(z)
\big\|_{H \ot T_z^*\C M} \\
& \q < K \cd \sup_{i \in \nn}\big( C_\tau \cd \ep 
+ \|d(t_i \cd \sqrt{\chi_i})(z)\|_{H \ot T_z^*\C M} \big) 
\leq K \cd \ep \cd ( C_\tau + 1 + C_\chi),
\end{split}
\end{equation}
where as above $x \in \C M$ and $z \in (\C M\sem \C K) \cap U_x$.

It is now a consequence of \eqref{eq:difstabII} and \eqref{eq:difstabIII} that the continuously differentiable section $\Psi(t) : \C M \to \sH$ vanishes at infinity. Indeed, we have that
\[
\sup_{z \in (\C M\sem \C K) \cap U_x}\big( \|\psi_x \ci \Psi(t)(z)\|_H + \|d(\psi_x \ci \Psi(t))(z)\|_{H \ot T_x^*\C M} \big)
< \ep \cd \big( 1 + K \cd (C_\tau + 1 + C_\chi) \big),
\]
for all $x \in \C M$, where the constants $C_\tau, C_\chi > 0$ and $K \in \nn$ are independent of $\ep >0$.

It has thus been verified that $\Psi : C_0^1(\C M,H^\infty) \to \Ga_0^1(\sH)$ is well-defined. Since it obviously respects the module structures over $C^1_0(\C M)$ this ends the proof of the proposition.
\end{proof}


\begin{thm}\label{t:difstab}
Let $\pi : \sH \to \C M$ be a Hilbert bundle of bounded geometry. Then the module of $C^1_0$-sections $\Ga_0^1(\sH)$ is an operator $*$-module over $C^1_0(\C M)$.
\end{thm}
\begin{proof}
It follows from Proposition \ref{p:isomet} and Proposition \ref{p:isosta} that the maps $\Phi : \Ga_0^1(\sH) \to C_0^1(\C M,H^\infty)$ and $\Psi : C_0^1(\C M,H^\infty) \to \Ga_0^1(\sH)$ are morphisms of modules over $C^1_0(\C M)$. 
Let now $s \in \Ga_0^1(\sH)$. Then
\[
(\Psi \ci \Phi)(s) = \Psi\big( \sum_{i=1}^\infty e_i (\psi_{x_i} \ci (s \cd \sqrt{\chi_i}) )\big) = \sum_{i=1}^\infty (\psi_{x_i}^{-1} \ci \psi_{x_i} \ci s) \cd \chi_i = s \cd \sum_{i=1}^\infty \chi_i = s.
\]
This shows that $\Psi \ci \Phi = 1_{\Ga_0^1(\sH)}$.

Let $P : C^1_0(\C M,H^\infty) \to C^1_0(\C M,H^\infty)$ be the composition $P := \Phi \ci \Psi$. It follows from the above that $P^2 = P$ and that $\Ga_0^1(\sH)$ is isomorphic as a module over $C^1_0(\C M)$ to the submodule $P C^1_0(\C M,H^\infty) \su C^1_0(\C M,H^\infty)$. It is therefore sufficient to prove that $P$ is completely bounded and selfadjoint.

Define the map $\sP : \C M \to \sL(H^\infty)$ by the formula
\[
\sP(x) : e_j \cd \xi \mapsto \sum_{i=1}^\infty e_i \cd \tau_{x_i,x_j}(x)(\xi) \cd
\sqrt{\chi_i \cd \chi_j}(x), 
\]
for all $x \in \C M$. Thus $\sP(x) \in \sL(H^\infty)$ is represented by the infinite block matrix where each block is an element of $\sL(H)$ and the block in position $(i,j)$ is $\tau_{x_i,x_j}(x) \cd \sqrt{\chi_i \cd \chi_j}(x)$. It follows that $\sP(x)$ is selfadjoint for all $x \in \C M$. Indeed, simply note that
\[
\big( \tau_{x_i,x_j}(x) \cd \sqrt{\chi_i \cd \chi_j}(x) \big)^*
= \tau_{x_j,x_i}(x) \cd \sqrt{\chi_i \cd \chi_j}(x)
\]
for all $i,j \in \nn$.

It is clear that $\sP : \C M \to \sL(H^\infty)$ corresponds to $P : C^1_0(\C M,H^\infty) \to C^1_0(\C M,H^\infty)$ under the bijective map of Proposition \ref{p:commanI}. It is therefore enough to show that $\sP$ is strongly differentiable with $\sup_{x \in \C M}\|(d\sP)(x)\|_\infty < \infty$.

The strong differentiability of $\sP : \C M \to \sL(H^\infty)$ follows from the strong differentiability of each transition map $\tau_{x_i,x_j} : U_{x_i} \cap U_{x_j} \to \sL(H)$ since the partition of unity $\{\chi_i\}$ is locally finite.

Let $C_\chi := \sup_{i,j \in \nn}\|d(\sqrt{\chi_i \cd \chi_j})\|_\infty < \infty$, let $C_\tau := \sup_{i,j \in \nn}\|d(\tau_{x_i,x_j})\|_\infty < \infty$, and let $K \in \nn$ be the multiplicity of the cover $\{\T{supp}(\chi_i)\}$. Let $x \in \C M$, let $\xi = \sum_{j =1}^\infty e_j \cd \xi_j \in H^\infty$, and compute as follows,
\[
\begin{split}
& \|(d\sP)(x)(\xi)\|_{H^\infty \ot T_x^*\C M}^2
= \sum_{i=1}^\infty \|\sum_{j=1}^\infty d(\tau_{x_i,x_j} \cd \sqrt{\chi_i \cd \chi_j})(x)(\xi_j)\|^2_{H \ot T_x^*\C M} \\
& \q \leq K^3 \cd \sup_{i,j \in \nn} \| d(\tau_{x_i,x_j} \cd \sqrt{\chi_i \cd \chi_j})(x)(\xi_j)\|^2_{H \ot T_x^*\C M} \\
& \q \leq K^3 \cd \sup_{i,j \in \nn}
\big( \|d(\tau_{x_i,x_j})(x)\|_\infty \cd \|\xi_j\|_H
+ \|\tau_{x_i,x_j}(x)(\xi_j) \ot (d\sqrt{\chi_i \cd \chi_j})(x)\|_{H \ot T_x^*\C M} \big)^2 \\
& \q \leq K^3 \cd ( C_\tau + C_\chi)^2 \cd \sup_{j \in \nn}\|\xi_j\|_H^2
\leq K^3 \cd ( C_\tau + C_\chi)^2 \cd \|\xi\|_{H^\infty}^2
.
\end{split}
\]
This shows that $\|(d\sP)(x)\|_\infty \leq \sqrt{K^3} \cd (C_\tau + C_\chi)$ for all $x \in \C M$ and the theorem is proved.


\end{proof}

\begin{remark}\label{r:openor}
The operator $*$-module norm on $\Ga_0^1(\sH)$ is given by
\[
\begin{split}
& \|\cd\|_1 : s \mapsto \|\Phi(s)\|_1 \\
& \, \, = 
\sup_{x \in \C M} \Big( \sum_{i=1}^\infty \|(\psi_{x_i} \ci s)(x)\cd \sqrt{\chi_i}(x)\|_H^2 + \sum_{i=1}^\infty\|d\big((\psi_{x_i} \ci s) \cd \sqrt{\chi_i}\big)(x)\|_{H\ot T_x^*\C M}^2 \Big)^{1/2} \\
& \, \, =
\sup_{x \in \C M} \Big( \|s(x)\|_{\sH_x}^2 + \sum_{i=1}^\infty\|d\big((\psi_{x_i} \ci s) \cd \sqrt{\chi_i}\big)(x)\|_{H\ot T_x^*\C M}^2 \Big)^{1/2}.
\end{split}
\]
for all $s \in M\big( \Ga_0^1(\sH) \big)$. It is important to notice that this norm depends on the choice of bounded partition of unity $\{\chi_i\}$ and the local trivializations $\{\psi_{x_i}\}$. As we shall see later on this dependency is however only up to a canonical unitary isomorphism. As expected the completely bounded hermitian form on $\Ga_0^1(\sH)$ is given by $\inn{s,t} : x \mapsto \inn{s(x),t(x)}_{\sH_x}$.
\end{remark}

\section{Image bundles of bounded geometry}\label{s:imabun}

Let $\C M$ be a smooth Riemannian manifold of dimension $N \in \nn$ and let $H$ be a separable Hilbert space of infinite dimension.

Recall from Section \ref{ss:comman} that $C^1_b(\C M,\sL(H))$ consists of the $*$-strongly differentiable maps $\al : \C M \to \sL(H)$ with
\[
\|\al\|_1 := \sup_{x \in \C M}\|\pi(\al)(x)\|_\infty < \infty,
\]
where $\pi(\al)(x)$ is the bounded operator defined by
\[
\pi(\al)(x) := \ma{cc}{
\al(x) & 0 \\
(d\al)(x) & \al(x) \ot 1
}  : H \op (H \ot T_x^*\C M) \to H \op (H \ot T_x^*\C M),
\]
for all $x \in \C M$. Remark that $C^1_b(\C M,\sL(H))$ becomes a Banach algebra when equipped with the norm $\|\cd\|_1$ and the pointwise algebraic operations. Indeed, the submultiplicativity of the norm $\|\cd\|_1$ follows since
\begin{equation}\label{eq:difsqrI}
\pi(\al \cd \be)(x) = \pi(\al)(x) \ci \pi(\be)(x),
\end{equation}
for all $\al,\be \in C^1_b(\C M,\sL(H))$ and all $x \in \C M$. Notice also that
\[
\frac{1}{2} \cd \big( \|\al(x)\|_\infty + \|(d\al)(x)\|_\infty \big) \leq \|\pi(\al)(x)\|_\infty \leq \|\al(x)\|_\infty + \|(d\al)(x)\|_\infty,
\]
for all $\al \in C^1_b(\C M,\sL(H))$ and all $x \in \C M$.
%
%

\begin{lemma}\label{l:difsqr}
Let $U \su \C M$ be an open set and let $\La \in C^1_b(U,\sL(H))$. Suppose that $\La(x) \in \sL(H)$ is positive and invertible for all $x \in U$ and that
\[
\T{sup}_{x \in U} \|\La^{-1}(x)\|_\infty < \infty.
\]
Then the map $\La^{-1/2} : x \mapsto \La(x)^{-1/2}$ is strongly differentiable with
\[
\sup_{x \in U}\|d(\La^{-1/2})(x)\|_\infty \leq
\sup_{x \in U} \|(d\La)(x)\|_\infty \cd \sup_{x \in U} \|\La^{-3/2}(x)\|_\infty.
\]
%
%
\end{lemma}
\begin{proof}

Let $\de := \sup_{x \in U}\|\La^{-1}(x)\|_\infty < \infty$. Notice that the spectrum $\T{Sp}_{\sL(H)}\big(\La(x)\big)$ is contained in $[\de^{-1},\infty)$ for all $x \in U$. In particular, the norm estimate
\begin{equation}\label{eq:imabunI}
\sup_{x \in U}\|(\la - \La(x))^{-1}\|_{\infty} \leq |\la - \de^{-1}|^{-1}
\end{equation}
holds for all $\la \in \cc\sem (0,\infty)$ by the continuous functional calculus. See for example \cite[Theorem 1.1.7]{Ped:CA}.

Let now $\la \in \cc\sem (0,\infty)$ be fixed. Notice that the map $(\la - \La)^{-1} : U \to \sL(H)$ is $*$-strongly differentiable with
\[
d(\la - \La)^{-1} = \big( (\la - \La)^{-1} \ot 1 \big)(d\La)(\la - \La)^{-1}
\]
As a consequence,
\begin{equation}\label{eq:imabunII}
\sup_{x \in U}\|d(\la - \La)^{-1}(x)\|_\infty
\leq |\la - \de^{-1}|^{-2} \cd \sup_{x \in U}\|d\La(x)\|_\infty.
\end{equation}

This shows that $(\la - \La)^{-1} \in C_b^1(U,\sL(H))$ for all $\la \in \cc\sem (0,\infty)$. Thus $\T{Sp}_{C_b^1(U,\sL(H))}(\La) \su (0,\infty)$. Since the function $z \mapsto z^{-1/2}$ is holomorphic on the half plane $\G{Re}(z) > 0$ and Banach algebras admit a holomorphic functional calculus it follows that $\La^{-1/2} \in C_b^1(U,\sL(H))$. See \cite[Theorem 10.27]{Rud:FA}.

Let $\C A := C_b^1(U,H)$ denote the Banach space of bounded continuously differentiable Hilbert space valued functions with the norm 
\[
s \mapsto \sup_{x \in U}\big( \|s(x)\|_H^2 + \|(ds)(x)\|_{H \ot T_x^*\C M}^2 \big)^{1/2}.
\]
Let $\C B := \Ga_b(U,H \ot T^*\C M)$ denote the Banach space of Hilbert space valued bounded continuous $1$-forms with the norm $\om \mapsto \sup_{x \in U}\|\om(x)\|_{H \ot T_x^*\C M}$. The de Rham differential then yields a bounded operator $d : \C A \to \C B$.

To obtain the estimate on the strong derivative $d(\La^{-1/2})$, let $\xi \in H$ be fixed, and note that
\[
\La^{-1/2}(\xi) = \frac{1}{\pi}\int_0^\infty \la^{-1/2}(\la + \La)^{-1}(\xi) \, d\la,
\]
where the integral converges absolutely in $\C A$. Indeed, this follows from 
the estimates
\[
\begin{split}
\frac{1}{\pi} \int_0^\infty \la^{-1/2} \sup_{x \in U}\|(\la + \La)^{-1}(\xi)(x)\|_\infty\, d\la 
& \leq \|\xi\|_H \cd \frac{1}{\pi} \int_0^\infty \la^{-1/2} (\la + \de^{-1})^{-1}\, d\la \\
& = \de^{1/2} \cd \|\xi\|_H
\end{split}
\]
and
\[
\begin{split}
& \frac{1}{\pi}\int_0^\infty \la^{-1/2}\sup_{x \in U}\|d(\la + \La)^{-1}(\xi)(x)\|_\infty \, d\la \\
& \q \leq
\frac{1}{\pi} \cd \sup_{x \in U}\|(d\La)(x)\|_\infty \cd \|\xi\|_H \cd 
\int_0^\infty \la^{-1/2} (\la + \de^{-1})^{-2} \, d\la \\
& \q \leq \sup_{x \in U}\|(d\La)(x)\|_\infty \cd \|\xi\|_H \cd \de^{3/2},
\end{split}
\]
which rely on \eqref{eq:imabunI} and \eqref{eq:imabunII}.

These considerations imply the desired bound on the strong derivative $d(\La^{-1/2})$.

%

\end{proof}

\begin{remark}
It is possible to give a more sophisticated proof of the first part of the above result using spectral invariance as investigated by Blackadard and Cuntz in the context of differentiable seminorms, see \cite[Proposition 3.12]{BlCu:DNS}.

The above result can also be proved by more direct methods (i.e. without reference to spectra and functional calculus). The proof then relies entirely on the integral formula $\La(x)^{-1/2} = \frac{1}{\pi}\int_0^\infty \la^{-1/2}(\la + \La(x))^{-1} \, d\la$ for the square root.
\end{remark}

Suppose from now on that $\C M$ has bounded geometry. Let $s > 0$ be a constant which satisfies the conditions in Definition \ref{d:bougeo}.

Let $P : C_0^1(\C M)^\infty \to C_0^1(\C M)^\infty$ be a completely bounded projection. By Proposition \ref{p:commanI} and the discussion in the beginning of this section $P$ corresponds precisely to a strongly differentiable map $\sP : \C M \to \sL(H)$ with $\sP(x)$ an orthogonal projection for each $x \in \C M$ and with
\[
\T{sup}_{x \in \C M}\|(d\sP)(x)\|_\infty < \infty.
\]

Choose an $r \in (0,s)$ such that
\[
r < \frac{1}{4 \cd \sup_{x \in \C M}\|g_{\phi_{x,s}}\|_\infty^{1/2} \cd \sup_{x \in \C M}\|d\sP(x)\|_\infty}.
\]
Notice that $\sup_{x \in \C M}\|g_{\phi_{x,s}}\|_\infty < \infty$ by Remark \ref{r:bougeoI}.

The constant $r > 0$ then satisfies the bounded geometry conditions for $\C M$. It will be fixed for the rest of this section. Consequently, for each $x \in \C M$, we use the notation $U_x := U_{x,r}$ and $\phi_x := \phi_{x,r} : U_x \to B_r(0) \su \rr^N$ for the associated normal chart.


\begin{lemma}\label{l:dispro}
Let $x \in \C M$ and let $y,z \in U_x$. Then $\|\sP(y) - \sP(z)\|_\infty < 1/2$.
\end{lemma}
\begin{proof}
It follows by Lemma \ref{l:opecon} that
\[
\begin{split}
\|\sP(y) - \sP(z)\|_\infty & \leq \|\phi_x(y)-\phi_x(z)\|_{\rr^N} \cd \sup_{w \in U_x}\|(d\sP)(w)\|_\infty \cd \sup_{w \in U_x}\|g_{\phi_x}(w)\|_\infty^{1/2} \\
& < 2r \cd \sup_{w \in U_x}\|(d\sP)(w)\|_\infty \cd \|g_{\phi_{x,s}}\|_\infty^{1/2}
< 1/2.
\end{split}
\]
This proves the lemma in question.
\end{proof}


\begin{lemma}\label{l:unidif}
Let $x \in \C M$. Then the operator 
\[
W_x(y) := \sP(y)(\sP(x)\sP(y)\sP(x))^{-1/2} : \sP(x) H \to \sP(y) H
\]
is well-defined and unitary for all $y \in U_x$. Furthermore, the map $\sI \ci W_x : U_x \to \sL\big(\sP(x)H,H\big)$ is $*$-strongly differentiable with
\[
\|d (\sI \cd W_x)(y)\|_\infty \, , \, \|d(\sI \cd W_{x})^* (y)\|_\infty \leq 
3 \cd \sqrt{2}
\cd \sup_{z \in U_x}\|(d\sP)(z)\|_\infty
\]
for each $y \in U_x$. Here $\sI(y) : \sP(y)H \to H$ denotes the inclusion.
\end{lemma}
\begin{proof}
To ease the notation, let $U := U_x$, let $P:= \sP(x)$, and let $W := W_x$. 

Let $y \in U$. By Lemma \ref{l:dispro}, $\|P - P\sP(y)P\|_\infty \leq \|P - \sP(y)\|_\infty < 1/2$. This implies that $P\sP(y)P : PH \to PH$ is invertible. Since it is also positive this shows that $W(y) : PH \to \sP(y)H$ is well-defined. It is not hard to see that $W(y)$ is unitary.

Let $\La := P \sP(\cd) P : U \to \sL(PH,PH)$. Notice then that
\[
\| \La^{-1}(y)\|_\infty \leq 
\sum_{i=0}^\infty \| P - \sP(y)\|^i_\infty < \sum_{i=0}^\infty (1/2)^i = 2,
\]
for all $y \in U$. Furthermore, $\sup_{y \in U}\|d\La(y)\|_\infty \leq \sup_{y \in U}\|(d\sP)(y)\|_\infty < \infty$. It thus follows from Lemma \ref{l:difsqr} that $\La^{-1/2}$ is strongly differentiable. Since $\sP : H \to H$ is strongly differentiable by assumption this implies the strong differentiablity of $\sI \cd W : U \to \sL(PH,H)$. 

The upper bound on the strong derivative of $\sI \cd W$ is also a consequence of Lemma \ref{l:difsqr}. Indeed,
\[
\begin{split}
\|d(\sI \cd W)(y)\|_\infty & \leq \|(d\sP)(y)\|_\infty \cd \|(\La^{-1/2})(y)\|_\infty + \|d(\La^{-1/2})(y)\|_\infty \\
& \leq \sqrt{2} \cd \sup_{z \in U}\|(d\sP)(z)\|_\infty  + \sup_{z \in U} \|(d\sP)(z)\|_\infty \cd \sup_{z \in U}\|\La^{-3/2}(z)\|_\infty \\
& \leq \sqrt{2} \cd 3 \cd \sup_{z \in U}\|(d\sP)(z)\|_\infty
\end{split}
\]
for all $y \in U$.

To see that the adjoint $(\sI W)^* : y \mapsto (\sI(y) W(y))^* \in \sL(H, PH)$ is strongly differentiable note that $\big( \sI(y) W(y) \big)^* = \La^{-1/2}(y) P \sP(y)$. The above proof now gives the desired result, including the upper bound on the strong derivative.
\end{proof}

Define the Hilbert space $\T{Im}(P)_x := \T{Im}(\sP(x))$ for each $x \in \C M$. Consider the disjoint union $\T{Im}(P) := \coprod_{x \in \C M} \T{Im}(P)_x$ of Hilbert spaces together with the projection $\pi : \T{Im}(P) \to \C M$, $\pi(\xi) = x \lrar \xi \in \T{Im}(P)_x$. The collection $\{U_x\}_{x \in \C M}$ is an open cover of $\C M$. For each $x \in \C M$ define the map
\[
\psi_x : \coprod_{y \in U_x} \sP(y)H \to U_x \ti \sP(x)H \q \psi_{x}(\xi) := \big(y,W_{x}(y)^*(\xi)\big) \, , \, \xi \in \sP(y)H.
\]

Define a basis $\sB$ for a topology on $\T{Im}(P)$ by
\[
V \in \sB \lrar \big( \pi(V) \su U_x \T{ for some }x\in \C M \, \, \T{ and }\, \, \psi_x(V) \su U_x \ti \sP(x)H \T{ is open } \big).
\]
The topology generated by this basis makes $\T{Im}(P)$ into a topological Hausdorff space such that the surjective map $\pi : \T{Im}(P) \to \C M$ is continuous, and such that the maps $\psi_x : \T{Im}(P)|_{U_x} \to U_x \ti \sP(x) H$ are homeomorphisms. Remark that all the fiber Hilbert spaces $\sP(x)H$ are unitarily isomorphic since $\C M$ is connected by assumption.
%

\begin{thm}\label{t:imabun}
The topological Hausdorff space $\T{Im}(P)$, the continuous surjective map $\pi : \T{Im}(P) \to \C M$, and the local trivializations $\psi_x : \T{Im}(P)|_{U_x} \to U_x \ti \sP(x)H$ gives $\T{Im}(P)$ the structure of a Hilbert bundle of bounded geometry.
\end{thm}
\begin{proof}
Let $x,y \in \C M$ and consider the transition map
\[
\tau_{x,y} : U_x \cap U_y \to \sL(\sP(y)H,\sP(x)H) \q \tau_{x,y}(z) := W_x(z)^* W_y(z).
\]
It follows from Lemma \ref{l:unidif} that $\tau_{x,y}(z)$ is unitary for all $z \in U_x \cap U_y$.

Notice now that $\tau_{x,y} = (\sI \ci W_x)^* (\sI \ci W_y)$, where as above $\sI(z) : \sP(z)H \to H$ denotes the inclusion for all $z \in \C M$. Lemma \ref{l:unidif} then yields that $\tau_{x,y}$ is strongly differentiable with
\[
\|d(\tau_{x,y})(z)\|_\infty \leq \|d(\sI W_x)^*(z)\|_\infty + \|d(\sI W_y)(z)\|_\infty
\leq 6 \cd \sqrt{2} \cd \sup_{w \in \C M}\|d(\sP)(w)\|_\infty.
\]
for all $z \in U_x \cap U_y$. This implies that $\T{sup}_{x,y \in \C M }\|d(\tau_{x,y})\|_\infty < \infty$ since $\T{sup}_{w \in \C M} \|d(\sP)(w)\| < \infty$ by assumption.
\end{proof}

\section{The Serre-Swan theorem}\label{s:serswa}
Throughout this section $\C M$ will be a manifold of bounded geometry and of dimension $N \in \nn$.

Recall from Theorem \ref{t:difstab} that the $C^1_0$-sections $\Ga_0^1(\sH)$ form an operator $*$-module over $C^1_0(\C M)$ for any Hilbert bundle $\sH \to \C M$ of bounded geometry.

\begin{prop}\label{p:secfun}
The assignment $\sH \mapsto \Ga_0^1(\sH)$, $\al \mapsto \Ga_0^1(\al)$ is a covariant functor from the category of Hilbert bundles of bounded geometry over $\C M$ to the category of operator $*$-modules over $C^1_0(\C M)$. The adjoints are related by the formula, $\Ga_0^1(\al^*) = \Ga_0^1(\al)^*$.
\end{prop}
\begin{proof}
Let $\al : \sH \to \sG$ be a morphism of Hilbert bundles of bounded geometry. Let us show that $\Ga_0^1(\al) : \Ga_0^1(\sH) \to \Ga_0^1(\sG)$ is completely bounded and has a completely bounded adjoint.

Suppose that the operator $*$-module structure on $\Ga_0^1(\sH)$ is determined by a bounded partition of unity $\{\chi_j\}$ with $\T{supp}(\chi_j) \su U_{x_j,r}$, and local trivializations $\psi_{x_j} : \sH|_{U_{x_j,r}} \to U_{x_j,r} \ti H$. Likewise, suppose that a bounded partition of unity $\{\si_i\}$ with $\T{supp}(\si_i) \su U_{y_i,s}$, and local trivializations $\rho_{y_i} : \sG|_{U_{y_i,s}} \to U_{y_i,s} \ti G$ determine the operator $*$-module structure on $\Ga_0^1(\sG)$. Here $H$ and $G$ are separable Hilbert spaces and $r,s>0$ are constants which satisfy the bounded geometry conditions for $\C M$.

Define the map $A : \C M \to \sL(H^\infty,G^\infty)$ as follows,
\[
A(x) : e_j \cd \xi \mapsto \sum_{i=1}^\infty e_i \cd \al_{y_i,x_j}(x)(\xi) \cd \sqrt{\si_i \cd \chi_j}(x),
\]
where $e_j \cd \xi \in H^\infty$ denotes the vector in the infinite direct sum of Hilbert spaces with $\xi \in H$ in position $j$ and zeroes elsewhere. Recall that $\al_{y_i,x_j} := \rho_{y_i} \ci \al \ci \psi_{x_j}^{-1} : U_{y_i,s} \cap U_{x_j,r} \to \sL(H,G)$ is notation for the transition maps associated with the morphism $\al : \sH \to \sG$.

The map $A$ then restricts to $\Ga_0^1(\al) : \Ga_0^1(\sH) \to \Ga_0^1(\sG)$ under the isomorphism of Proposition \ref{p:commanI}. Notice that we have tacitly identified $\Ga_0^1(\sH)$ with a direct summand in $C^1_0(\C M,H^\infty)$ and $\Ga_0^1(\sG)$ with a direct summand in $C^1_0(\C M,G^\infty)$ using the maps corresponding to our choice of partitions of unities and local trivializations. See Section \ref{s:difstab}.
%

It is therefore enough to show that $A : \C M \to \sL(H^\infty,G^\infty)$ is $*$-strongly differentiable with $\|A\|_1 = \sup_{x \in \C M}\|\pi(A)(x)\|_\infty < \infty$. See Section \ref{ss:comman} for the definition of the norm $\|\cd\|_1 : C^1_b\big(\C M,\sL(H^\infty,G^\infty)\big) \to [0,\infty)$.

Notice that $A(x)$ is an infinite block matrix of elements in $\sL(H,G)$ for each $x \in \C M$. The block in position $(i,j)$ is the bounded operator $\al_{y_i,x_j}(x) \cd \sqrt{\si_i\cd \chi_j}(x)$. Since $\al_{y_i,x_j} \cd \sqrt{\si_i \cd \chi_j} : \C M \to \sL(H,G)$ is $*$-strongly differentiable for each $i,j \in \nn$ it follows by the local finiteness of the partitions of unity $\{\si_i\}$ and $\{\chi_j\}$ that $A : \C M \to \sL(H^\infty,G^\infty)$ is $*$-strongly differentiable.

Since $\al : \sH \to \sG$ is a morphism of Hilbert bundles of bounded geometry, there exists a constant $C_\al > 0$ such that
\[
\|\pi(\al_{y_i,x_j})(x)\|_\infty \leq \|\al_{y_i,x_j}(x)\|_\infty + \|d(\al_{y_i,x_j})(x)\|_\infty
\leq C_\al
\]
for all $i,j \in \nn$ and all $x \in U_{y_i,s} \cap U_{x_j,r}$. See the discussion in the beginning of Section \ref{s:imabun}.

Furthermore, since the partitions of unity $\{\si_i\}$ and $\{\chi_j\}$ are bounded, there exists a constant $C_{\si,\chi} > 0$ such that
\[
\sup_{i,j \in \nn}\|\pi(\sqrt{\si_i \cd \chi_j})(x)\|_\infty \leq \sup_{i,j \in \nn}\big( \|\sqrt{\si_i \cd \chi_j}(x)\|_\infty + \|d(\sqrt{\si_i \cd \chi_j})(x)\|_\infty\big)
\leq C_{\si,\chi},
\]
for all $x \in \C M$, where $\sqrt{\si_i \cd \chi_j}$ is perceived as an element in $C^1_b(\C M,\sL(H))$ for all $i,j \in \nn$.

These two computations imply that
\[
\sup_{i,j \in \nn}\|\pi(\al_{y_i,x_j} \cd \sqrt{\si_i \cd \chi_j})(x)\|_\infty
\leq \|\pi(\al_{y_i,x_j})(x)\|_\infty \cd \|\pi(\sqrt{\si_i \cd \chi_j})(x)\|_\infty
\leq C_\al \cd C_{\si,\chi}
\]
for all $x \in \C M$, where the identity in \eqref{eq:difsqrI} has been applied.

Let $K \in \nn$ and $L \in \nn$ denote the multiplicites of the partitions of unity $\{\chi_j\}$ and $\{\si_i\}$ respectively. 
%

Let $x \in \C M$ and let $\xi = \sum_{j=1}^\infty e_j \xi_j \in H^\infty \op (H^\infty \ot T_x^*\C M$). Compute then as follows
\[
\begin{split}
& \| \pi(A)(x)(\xi) \|^2_{G^{\infty} \op (G^\infty \ot T_x^*\C M)} 
= \sum_{i=1}^\infty \|\sum_{j=1}^\infty \pi(\al_{y_i,x_j} \cd \sqrt{\si_i \cd \chi_j})(x)(\xi_j)\|_{G \op (G \ot T_x^*\C M)}^2 \\
& \q \leq L \cd K^2 \cd \sup_{i,j \in \nn} \|\pi(\al_{y_i,x_j} \cd \sqrt{\si_i \cd \chi_j})(\xi_j)\|_{G \op (G \ot T_x^*\C M)}^2 \\
& \q \leq L \cd (K \cd C_\al \cd C_{\si,\chi})^2 \cd \sup_{j \in \nn}\|\xi_j\|_{H \op (H \ot T_x^*\C M)}^2 \\
& \q \leq L \cd (K \cd C_\al \cd C_{\si,\chi})^2 \cd \|\xi\|_{H^\infty \op (H^\infty \ot T_x^*\C M)}^2.
\end{split}
\]
This shows that $\|A\|_1 = \sup_{x \in \C M}\|\pi(A)(x)\|_\infty < \infty$ as desired.
%
%
\end{proof}

\begin{remark}
As noted in Remark \ref{r:openor}, the operator $*$-module structure of $\Ga_0^1(\sH)$ depends on the choice of a bounded partition of unity and local trivializations. It follows however from the above proposition that this dependency is only up to a canonical unitary isomorphism. Indeed, the identity map $1_{\sH} : \sH \to \sH$ is a morphism of Hilbert bundles of bounded geometry, see Remark \ref{r:idemor}.
\end{remark}

\begin{lemma}\label{l:serswaI}
The functor $\Ga_0^1 : \G{Hilb}_{\C M} \to \G{Op^*Mod}_{C^1_0(\C M)}$ is essentially surjective.
\end{lemma}
\begin{proof}
Recall that $\Ga_0^1$ is essentially surjective when each operator $*$-module $X$ over $C^1_0(\C M)$ is isomorphic to an operator $*$-module of the form $\Ga_0^1(\sH)$.

Let $X$ be an operator $*$-module over $C^1_0(\C M)$. Without loss of generality we may assume that $X = \sP C^1_0(\C M,H)$, where $H$ is a separable infinite dimensional Hilbert space and $\sP : \C M \to \sL(H)$ is a strongly differentiable projection valued map with $\sup_{x \in \C M}\|(d\sP)(x)\|_\infty < \infty$. See Proposition \ref{p:starie} and Proposition \ref{p:commanI}.

By Theorem \ref{t:imabun} the associated image bundle $\T{Im}(P) \to \C M$ is a Hilbert bundle of bounded geometry. Recall in this respect that the fiber is $\T{Im}(P)_x = \sP(x)H$ for each $x \in \C M$ and that the local trivializations are given by $\psi_x = W_x^* : \T{Im}(P)|_{U_x} \to U_x \ti \sP(x)H$ for all $x \in \C M$. See Lemma \ref{l:unidif}. Here $\{U_x\}_{x \in \C M}$ is an open cover of normal coordinate neighborhoods associated with some suitable $r > 0$ which satisfies the bounded geometry condition for $\C M$. 

By definition, the operator $*$-module structure on $\Ga_0^1\big(\T{Im}(P)\big)$ is determined by a bounded partition of unity $\{\chi_i\}$ and local trivializations $\psi_{x_i} = W_{x_i}^* : \T{Im}(P)|_{U_{x_i}} \to U_{x_i} \ti \sP(x_i)H$. See Definition \ref{d:boupar} for the precise conditions on $\{\chi_i\}$.

Define the map $A : \C M \to \sL(H,H^\infty)$ by
\[
A(x) : \xi \mapsto \sum_{i=1}^\infty e_i \cd W_{x_i}^*(x) \sP(x)(\xi) \cd \sqrt{\chi_i}(x),
\]
for all $x \in \C M$ and all $\xi \in H$. We can think of $A(x)$ as an infinite column of bounded operators from $H$ to $H$. Since each of the entries is $*$-strongly differentiable by Lemma \ref{l:unidif} and the partition of unity $\{\chi_i\}$ is locally finite it follows that $A$ is $*$-strongly differentiable. Let us prove that $\|A\|_1 < \infty$. See Section \ref{ss:comman} for the definition of the norm $\|\cd\|_1$.

By Lemma \ref{l:unidif} there exists a constant $C_\tau > 0$ such that
\begin{equation}\label{eq:esssurI}
\|\pi(W_{x_i}^* \cd \sP)(x)\|_\infty \leq C_\tau,
\end{equation}
for all $i \in \nn$ and $x \in U_{x_i}$. Likewise, since the partition of unity $\{\chi_i\}$ is bounded there exists a constant $C_\chi > 0$ such that
\begin{equation}\label{eq:esssurII}
\sup_{i \in \nn}\|\pi(\sqrt{\chi_i})(x)\|_\infty \leq C_{\chi}
\end{equation}
for all $x \in \C M$, where $\sqrt{\chi_i}$ is perceived as an element in $C^1_b(\C M,\sL(H))$ for all $i \in \nn$.

Let $K \in \nn$ denote the multiplicity of $\{\chi_i\}$. Let $x \in \C M$, let $\xi \in H \op (H \ot T_x^*\C M)$, and note that
\[
\begin{split}
& \|\pi(A)(x)(\xi)\|^2_{H^\infty \op (H^\infty \ot T_x^*\C M)}
= \sum_{i=1}^\infty\|\pi(W_{x_i}^* \sP \cd \sqrt{\chi_i})(x)(\xi)\|_{H \op (H \ot T_x^*\C M)}^2 \\
& \q \leq K \cd \sup_{i \in \nn} \|\pi(W_{x_i}^* \sP \cd \sqrt{\chi_i})(x)(\xi)\|_{H \op (H \ot T_x^*\C M)}^2
\leq K \cd C_\tau^2 \cd C_\chi^2 \cd \|\xi\|_{H \op (H \ot T_x^*\C M)}^2,
\end{split}
\]
where we have applied the identity
\[
\pi(W_{x_i}^* \sP \cd \sqrt{\chi_i})(x) =
\pi(W_{x_i}^* \sP)(x) \ci \pi(\sqrt{\chi_i})(x), 
\]
as well as the estimates in \eqref{eq:esssurI} and \eqref{eq:esssurII}. This shows that $\|A\|_1 < \infty$.

It follows from these considerations and Proposition \ref{p:commanI} that $A$ induces a morphism $C^1_0(\C M,H) \to C^1_0(\C M,H^\infty)$ of hermitian operator modules. It is not hard to see that this map restricts to the morphism
\[
\al : \sP C^1_0(\C M,H) \to \Ga_0^1(\T{Im}(P)) \q \al(\xi)(x) = \xi(x)
\]
of hermitian operator modules. Here we have identified $\Ga_0^1(\T{Im}(P))$ with a direct summand in $C^1_0(\C M,H^\infty)$ as in Section \ref{s:difstab}.

This establishes that $X$ is (unitarily) isomorphic to $\Ga_0^1(\T{Im}(P))$ in the category of operator $*$-modules. Indeed, the adjoint of $\al$ is given by
\[
\al^* : \Ga_0^1(\T{Im}(P)) \to \sP C^1_0(\C M,H) \q \al^*(s)(x) = s(x).
\]

The functor $\Ga_0^1 : \G{Hilb}_{\C M} \to \G{Op^*Mod}_{C^1_0(\C M)}$ is thus essentially surjective.
\end{proof}

\begin{lemma}\label{l:serswaII}
The functor $\Ga_0^1 : \G{Hilb}_{\C M} \to \G{Op^*Mod}_{C^1_0(\C M)}$ is fully faithful.
\end{lemma}
\begin{proof}
Recall that $\Ga_0^1$ is fully faithful when the map
\[
\Ga_0^1 : \T{Mor}(\sH,\sG) \to \T{Mor}(\Ga_0^1(\sH),\Ga_0^1(\sG))
\]
if bijective whenever $\sH$ and $\sG$ are Hilbert bundles of bounded geometry.

The injectivity part is easy to prove and well-known. 

Thus, let $\sH$ and $\sG$ be two Hilbert bundles of bounded geometry over $\C M$ and let $\al : \Ga_0^1(\sH) \to \Ga_0^1(\sG)$ be a morphism of operator $*$-modules.

Let $\{\al_x\}_{x \in \C M}$ denote the unique set of maps such that $\al_x : \sH_x \to \sG_x$ and $\al(s)(x) = \al_x(s(x))$ for all $x \in \C M$ and all sections $s \in \Ga_0^1(\sH)$. It is enough to show that each transition map $\al_{y,x} : U_y \cap U_x \to \sL(H,G)$ lies in $C^1_b(U_y \cap U_x, \sL(H,G))$ and that $\sup_{y,x \in \C M} \|\al_{y,x}\|_1 < \infty$. See Definition \ref{d:hilbou} and the discussion in the beginning of Section \ref{s:imabun}.

We may suppose that the operator $*$-module structures on $\Ga_0^1(\sH)$ and $\Ga_0^1(\sG)$ are given by a bounded partition of unity $\{\chi_i\}$ and local trivializations $\psi_{x_j} : \sH|_{U_{x_j}} \to U_{x_j} \ti H$ and $\rho_{x_i} : \sG|_{U_{x_i}} \to U_{x_i} \ti G$. See Section \ref{s:difstab}. Remark that the open covers of $\C M$ by normal coordinate neighborhoods and the bounded partitions of unity may be chosen to agree by Proposition \ref{p:secfun}.

The projection $\sP : \C M \to \sL(H^\infty)$ associated with $\Ga_0^1(\sH)$ is thus given by
\[
\sP(x) : e_j \cd \xi \mapsto \sum_{i=1}^\infty e_i \cd \tau_{x_i,x_j}(x)(\xi) \cd \sqrt{\chi_i \cd \chi_j}
\]
for all $x \in \C M$, $\xi \in H$, and $j \in \nn$. See the proof of Theorem \ref{t:difstab}.

Likewise, the module homomorphism $\Phi : \Ga_0^1(\sH) \to \sP C^1_0(\C M,H^\infty)$ is given by
\[
\Phi(s)(x) = \sum_{i = 1}^\infty e_i \cd \psi_{x_i}(s \cd \sqrt{\chi_i}),
\]
for all $x \in \C M$, and all $s \in \Ga_0^1(\sH)$.

For each $j \in \nn$ and each $\xi \in H$, let $\xi_{x_j} \in \Ga_0^1(\sH)$ be defined by $\xi_{x_j}(x) := \psi_{x_j}^{-1}(x,\xi \cd \sqrt{\chi_j}(x))$ for all $x \in \C M$. 

Let $\xi \in H$ and let $j \in \nn$. Notice then that
\begin{equation}\label{eq:serswaI}
\begin{split}
\|\xi_{x_j}\|_{\Ga_0^1(\sH)}
& = \|\Phi(\xi_{x_j})\|_{C^1_0(\C M,H^\infty)}
= \|\sum_{i=1}^\infty e_i \cd \tau_{x_i,x_j}(\xi) \cd \sqrt{\chi_i \cd \chi_j}\|_{C^1_0(\C M,H^\infty)} \\
& = \|\sP(e_j \cd \xi)\|_{C^1_0(\C M,H^\infty)}
\leq \|\sP\|_{\T{cb}} \cd \|\xi\|_H.
\end{split}
\end{equation}

To continue, remark that
\begin{equation}\label{eq:serswaII}
\chi_j(x) \cd \al_{x_j,x_j}(\xi)(x) = \sqrt{\chi_j}(x) \cd \rho_{x_j}\big( \al(\xi_{x_j})(x) \big)
\end{equation}
for all $x \in \C M$. This shows that $\chi_j \cd \al_{x_j,x_j} : \C M \to \sL(H,G)$ is strongly differentiable. Since a similar argument applies to the adjoint we conclude that $\chi_j \cd \al_{x_j,x_j}$ is $*$-strongly differentiable.

Now, apply \eqref{eq:serswaI} and \eqref{eq:serswaII} to compute as follows,
\begin{equation}\label{eq:serswaIII}
\begin{split}
\|\chi_j \cd \al_{x_j,x_j}(\xi)\|_1
& = \|\sqrt{\chi_j} \cd \rho_{x_j}\big( \al(\xi_{x_j}) \big)\|_1
\leq \|\al(\xi_{x_j})\|_{\Ga_0^1(\sG)} \leq \|\al\|_{\T{cb}} \cd \|\xi_{x_j}\|_{\Ga_0^1(\sH)} \\
& \leq \|\al\|_{\T{cb}} \cd \|\sP\|_{\T{cb}} \cd \|\xi\|_H.
\end{split}
\end{equation}

Since the partition of unity $\{\chi_j\}$ has finite multiplicity, there exists a constant $\de_{\chi} > 0$ such that for every $x \in \C M$ there exists a neighborhood $W_x$ and a $j \in \nn$ with $\chi_j(y) \geq \de_{\chi}$ for all $y \in W_x$.

Let $x \in \C M$ and choose a neighborhood $W_x$ and a $j \in \nn$ as above.
Let $C_\chi := \sup_{j \in \nn}\|d\chi_j\|_\infty$. Notice that
\begin{equation}\label{eq:serswaIV}
\|\pi(\chi_j^{-1})(y)\|_\infty
= \big\| 
\ma{cc}{
\chi_j^{-1}(y) & 0 \\
-(\chi_j^{-2} d\chi_j)(y) & \chi_j^{-1}(y)
}
\big\|_\infty
\leq \de_{\chi}^{-1} + \de_{\chi}^{-2} \cd C_\chi
\end{equation}
for all $y \in W_x$, where $\chi_j^{-1}$ is perceived as an element in $C_b^1(W_x,\sL(G))$. To ease the notation, put $\de_{\chi}^{-1} \cd (1 + \de_{\chi}^{-1} \cd C_\chi) := C_{\de,\chi}$.

Combining \eqref{eq:serswaIII} and \eqref{eq:serswaIV} we obtain that
\[
\begin{split}
& \|\pi(\al_{x_j,x_j})(y)\ma{c}{\xi \\ 0}\|_{G \op (G \ot T_y^*\C M)}
\leq
\|\pi(\chi_j \cd \al_{x_j,x_j})(y)\ma{c}{\xi \\ 0}\|_{G \op (G \ot T_y^*\C M)}
\cd C_{\de,\chi} \\
& \q \leq \|\chi_j \cd \al_{x_j,x_j}(\xi)\|_1 \cd C_{\de,\chi}
\leq \|\al\|_{\T{cb}} \cd \|\sP\|_{\T{cb}} \cd \|\xi\|_H \cd C_{\de,\chi},
\end{split}
\]
for all $y \in W_x$. This shows that
\begin{equation}\label{eq:serswaV}
\|\pi(\al_{x_j,x_j})(y)\| \leq 2 \cd \|\al\|_{\T{cb}} \cd \|\sP\|_{\T{cb}} \cd C_{\de,\chi},
\end{equation}
for all $y \in W_x$. 

Let $C_\tau := \sup_{x,y \in \C M}\|\tau_{x,y}\|_1$ and $C_\si := \sup_{x,y \in \C M}\|\si_{x,y}\|_1$, where $\tau_{x,y} : U_x \cap U_y \to \sL(H)$ and $\si_{x,y} : U_x \cap U_y \to \sL(H)$ denote the transition maps of the bundles $\sH$ and $\sG$ respectively.

Let $x,y \in \C M$ and let $z \in U_x \cap U_y$. Choose a neighborhood $W_z \su U_x \cap U_y$ of $z$ and a $j \in \nn$ such that $\chi_j^{-1}(w) \geq \de_\chi$ for all $w \in W_z$. The estimate in \eqref{eq:serswaV} now implies that
\[
\begin{split}
\|\pi(\al_{y,x})(w)\|_\infty & \leq 
\|\pi(\si_{y,x_j})(w)\|_\infty \cd \|\pi(\al_{x_j,x_j})(w)\|_\infty \cd
\|\pi(\tau_{x_j,x})(w)\|_\infty \\
& \leq C_\si \cd 2 \cd \|\al\|_{\T{cb}} \cd \|\sP\|_{\T{cb}} \cd C_{\de,\chi} \cd C_\tau
\end{split}
\]
for all $w \in W_z$. See also \eqref{eq:difsqrI}. This proves the proposition.
\end{proof}

\begin{thm}\label{t:serswa}
The covariant functor $\Ga_0^1 : \G{Hilb}_{\C M} \to \G{Op^*Mod}_{C^1_0(\C M)}$ is an equivalence of categories.
\end{thm}
\begin{proof}
It is enough to show that $\Ga_0^1$ is fully faithful and essentially surjective, see \cite[Chapter IV, Theorem 1]{Ma:CWM}. But this was already proved in Lemma \ref{l:serswaI} and Lemma \ref{l:serswaII}.
\end{proof}

\bibliographystyle{amsalpha-lmp}

\providecommand{\bysame}{\leavevmode\hbox to3em{\hrulefill}\thinspace}
\providecommand{\MR}{\relax\ifhmode\unskip\space\fi MR }
\providecommand{\MRhref}[2]{%
  \href{http://www.ams.org/mathscinet-getitem?mr=#1}{#2}
}
\providecommand{\href}[2]{#2}


\end{document}